\documentclass[leqno,11pt]{amsart}
\usepackage{amssymb, amsmath}
\usepackage{color}

\let\pa=\partial
\let\e=\epsilon
\let\ve=\varepsilon

\let\f=\frac

\let\D=\Delta

\let\wh=\widehat

\def\PP{\Bbb{P}}
\def\na{\nabla}
\def\ga{\gamma}
\def\la{\lambda}
\def\d{\delta}

\def\h{{\rm h}}
\def\T{\frak{T}}
\newcommand{\w}[1]{\langle {#1} \rangle}

\def\dive{\mathop{\rm div}\nolimits}

\def\cB{{\mathcal B}}

\def\cF{{\mathcal F}}

\def\cS{{\mathcal S}}
\def\Z{\mathop{\mathbb Z\kern 0pt}\nolimits}
\def\N{\mathop{\mathbb N\kern 0pt}\nolimits}
\def\Q{\mathop{\mathbb Q\kern 0pt}\nolimits}
\def\R{\mathop{\mathbb R\kern 0pt}\nolimits}

\def\eqdefa{\buildrel\hbox{\footnotesize def}\over =}
\newcommand{\andf}{\quad\hbox{and}\quad}
\newcommand{\with}{\quad\hbox{with}\quad}

\newcommand{\beq}{\begin{equation}}
\newcommand{\eeq}{\end{equation}}
\newcommand{\ben}{\begin{eqnarray}}
\newcommand{\een}{\end{eqnarray}}
\newcommand{\beno}{\begin{eqnarray*}}
\newcommand{\eeno}{\end{eqnarray*}}

\newtheorem{defi}{Definition}[section]
\newtheorem{thm}{Theorem}[section]
\newtheorem{lem}{Lemma}[section]
\newtheorem{rmk}{Remark}[section]

\newtheorem{prop}{Proposition}[section]

\numberwithin{equation}{section}
\begin{document}

\title[Refined analyticity radius of NS equations ]
{On the refined  analyticity radius of 3-D generalized Navier-Stokes equations}

\author[Dong Li]{Dong Li}
\address[Dong Li]{Department of Mathematics, the  University
of Hong Kong, Hong Kong, China}%
\email{mathdl@hku.hk}

\author[Ping Zhang]{Ping Zhang}
\address[Ping Zhang]{Academy of Mathematics $\&$ Systems Science
	and Hua Loo-Keng Center for Mathematical Sciences, Chinese Academy of
	Sciences, Beijing 100190, CHINA, and School of Mathematical Sciences,
	University of Chinese Academy of Sciences, Beijing 100049, China.} \email{zp@amss.ac.cn}

\date{\today}

\begin{abstract}
{We analyze the instantaneous growth of analyticity radius for three
 dimensional  generalized Navier-Stokes equations.}
 For the subcritical $H^{\gamma}(\mathbb R^3)$
case with $\ga>\f12,$ we
 prove that there exists a positive time $t_0$ so that for any $t\in]0, t_0]$, the radius of analyticity of the solution $u$ satisfies the pointwise-in-time lower bound
 $${\rm rad}(u)(t)\geq \sqrt{(2\ga-1)t\bigl(|\ln t|+\ln|\ln t|+K_t\bigr)},$$
{ where $K_t \to \infty$ as $t\to 0^+$}.
 This in particular gives a nontrivial improvement of  the previous result  by Herbst and Skibsted in \cite{HS} for the case $\ga\in ]1/2,3/2[$ and
  also settles the  {decade-long} open question in \cite{HS}, namely, whether or not
 $\liminf_{t\to 0^+}\f{{\rm rad}(u)(t)}{\sqrt{t|\ln t|}}\geq \sqrt{2\ga-1}$ for all $\ga\ge \f32.$ In the critical case
 $H^{\frac 12}(\mathbb R^3)$ we prove that there exists $t_1>0$ so that for any $t\in ]0, t_1],$ ${\rm rad}(u)(t)\geq \la(t)\sqrt{t}$ with
 $\la(t)$ satisfying $\lim_{t\to 0^+}\la(t)=\infty.$

\end{abstract}
\maketitle



\setcounter{equation}{0}
\section{Introduction}
In this paper, we consider the instantaneous growth of analyticity radius for the solutions to the following  $3$-D  generalized
 Navier-Stokes equations in $\R^+\times\R^3:$
\begin{equation*}
(GNS)\qquad \left\{\begin{array}{l}
\displaystyle \pa_t u-\D u=Q(u,u), \qquad (t,x)\in\R^+\times\R^3, \\
\displaystyle  u|_{t=0}=u_0.
\end{array}\right.
\end{equation*}
{Here $u=(u^1,u^2,u^3):\, \mathbb R^+ \times \mathbb R^3 \to \mathbb R^3$
denotes the velocity of the fluid under study. The viscosity preceding the Laplacian term is set to be one.
Throughout this paper we shall denote by $Q=(Q^1, Q^2, Q^3)$ any bilinear map of the form:
\beq\label{S1eq1}
Q^j(u,v)\eqdefa \sum_{k,\ell,m=1}^3 q^{j,m}_{k,\ell}(D)\pa_m(u^k v^\ell),
\eeq
where} $q^{j,m}_{k,\ell}(D)$ is a Fourier multiplier with symbol $q^{j,m}_{k,\ell}(\xi)\eqdefa\sum_{n,p=1}^3\alpha_{k,\ell}^{j,m,n,p}\f{\xi_n\xi_p}{|\xi|^2},$
and $\alpha_{k,\ell}^{j,m,n,p}$ are real numbers. {The precise numerical values of
$\alpha_{k, \ell}^{j,m,n,p}$ will not play any role in our analysis. Henceforth from a practical point of view it is
often useful
to regard $Q(u,v)$ as
\begin{align}
Q(u,v) =\mathcal R \partial (uv),
\end{align}
where $\mathcal R$ denotes a general Riesz transform.  Using this abstraction it is easy to deduce scaling transformations
associated with $(GNS)$.
}
{Namely if $u=u(t,x)$ is a smooth solution to $(GNS)$, then for  $\lambda>0$,
\begin{align}
u_{\lambda}(t,x) \eqdefa \lambda u(\lambda^2 t, \lambda x)
\end{align}
forms a one-parameter family of smooth solutions to $(GNS)$.  The  homogeneous
space $\dot H^{\frac 12}(\mathbb R^3)$ is critical
in the sense that $\|u_{\lambda}(t,\cdot)\|_{\dot H^{\frac 12}} = \|u(t,\cdot)\|_{\dot H^{\frac 12}}$ for any $\lambda>0$.
By a slight generalization we designate the inhomogeneous spaces $H^s(\mathbb R^3)$, $s=\frac 12$, $s>\frac 12$ as critical and subcritical
spaces respectively.
}

The motivation for us to study the system $(GNS)$ comes from the following classical 3-D incompressible Navier-Stokes equations:
\begin{equation*}
(NS)\qquad \left\{\begin{array}{l}
\displaystyle \pa_t u+u\cdot\nabla u-\D u=-\nabla P, \qquad (t,x)\in\R^+\times\R^3, \\
\displaystyle \dive u = 0, \\
\displaystyle  u|_{t=0}=u_0,
\end{array}\right.
\end{equation*}
where $u$ stands for the  fluid  velocity and
$P$ for the scalar pressure function, which guarantees the divergence free condition of the velocity field.
In fact, by applying
 Leray projection operator, $\PP={\rm I}+\na(-\D)^{-1}\dive,$ to $(NS),$ we obtain equations of the type $(GNS).$ One
 may check pages 206-207 of \cite{bcd} for a { motivating discussion of the system $(GNS)$. See
 also Chapter 5 therein for an extensive review of classical wellposedness results for $(GNS)$.}

In the { magnificent seminal} paper \cite{lerayns}, Leray proved the global existence of weak solution and local existence of strong solution to $(NS).$
It is well-known that strong solutions of $(NS)$ are in fact analytic in both space and time variables (see  \cite{lemarie3} for instance).
In fluid mechanics, the space analyticity radius of solutions to Navier-Stokes equations
yields a Kolmogrov type length scale encountered in turbulence theory, one may check \cite{BJMT, BS07,DT95, F97,HKR90,Ku99} and the references therein for more details.

 Mathematically, the study of analyticity of solutions to the Navier-Stokes equations  goes
back to Masuda in \cite{Ma67}, where the authors used complex-analytic techniques to investigate the
 analyticity in both space and time for the solutions of 2-D Navier-Stokes equations  in a
bounded domain with Dirichlet boundary {conditions}.
Foias and Temam \cite{foiastemam} introduced the notion of Gevrey norm, which
allows one to study analyticity properties of solutions via energy method. In particular, they \cite{foiastemam}
 proved the analyticity of periodic solutions of $(NS)$ in space and time with initial data $u_0\in H^1(\Bbb{T}^3)$ (see
 also \cite{FMRT}).
 Gruji\'c and Kukavica \cite{GK98} investigated the analyticity radius of the solution to $(NS)$ with initial data in $L^p$ for $p$
 greater than the space dimensions.
 The related result was later
extended by the authors in \cite{cheminleray,katomasuda,lemarie2,lemarie1} to show that: there exists a positive time $T$
so that
\beno
\int_{\R^3 }|\xi| \bigl(\sup_ {t\leq T}
e^{     \sqrt t|\xi|} |\wh u (t,\xi)| \bigr)^2\,d\xi + \int_0^T \int_{\R^3 }|\xi|^3 \bigl(
e^{     \sqrt t|\xi|} |\wh u (t,\xi)| \bigr)^2\,d\xi dt<\infty\,.
\eeno
{This
 in particular implies the Fujita-Kato solution of $(NS),$  which was constructed by Fujita and Kato in \cite{fujitakato},   with initial data $u_0\in \dot H^{\f12}(\R^3)$ is analytic for
any positive time $t.$} One may check related results in the survey book \cite{lemarie3}.

We remark that in the previous works \cite{CGZ3,foiastemam,HS}, the authors used Gevrey norm of the form
$\|e^{\sqrt{{\rm r}(t)}{|D|}}u(t)\|_X$ with a $L^2$ based Sobolev space $X.$
In \cite{lemarie2}  Lemari\'e-Rieusset  studied
Gevrey regularities of the solution $u$ to $(NS)$ in the $L^p$ framework. One may check \cite{BBT, HZ1, Zhang22}
for {more} recent development in this direction.

Before proceeding, we recall the definition of Sobolev spaces from \cite{bcd}:

\begin{defi}\label{S1def1}
{\sl \begin{enumerate}
\item[(1)] For $s\in\R,$ we define the inhomogeneous Sobolev space $H^{s}(\R^3)$ to be
the space of those tempered distributions $f$ which satisfy
\beno
\|f\|_{H^s}\eqdefa \|\w{\xi}^s\widehat{f}(\xi)\|_{L^2}<\infty,\eeno
where $\widehat{f}$  denotes the Fourier transform of $f.$ Here and  in  all that follows, we always denote  the quantity $\w{\xi} \eqdefa \bigl(1 + |\xi|^2\bigr)^{\f12}.$

\item[(2)] For $s\in\R,$ we define the homogeneous Sobolev space $\dot H^{s}(\R^3)$ to be
the space of those homogeneous distributions $f$ which satisfy
\beno
\|f\|_{\dot H^s}\eqdefa \||{\xi}|^s\widehat{f}(\xi)\|_{L^2}<\infty.\eeno

\end{enumerate}
}
\end{defi}

{From a scaling perspective, the $\mathcal O(\sqrt t)$-radius of analyticity of the solution to $(NS)$ seems to be
optimal since it almost fully utilizes the heat kernel. Thus it is somewhat surprising that
Herbst and Skibsted \cite{HS} proved the following sharpened result}:

\begin{thm}[Theorem 1.3 of \cite{HS}]\label{thHS}
{\sl Suppose $u_0\in H^\ga$ for some $\ga\in ]1/2,3/2[.$ Then the system $(NS)$ with initial data $u_0$
has a unique local solution $u$ on $[0,T].$ Let $\ve\in ]0,2\ga-1[.$ Then there exist constant $t_0=t_0(\ve,\ga,\|u_0\|_{H^\ga})\in
]0,T]$ and $C=C(\ve,\ga,\|u_0\|_{H^\ga})>0$ such that
\beq\label{S1eq12}
\bigl\|e^{\sqrt{2\ga-1-\ve}\sqrt{t|\ln t|}|D|}u(t)\bigr\|_{H^\ga}\leq Ct^{\f14+\ve-\f{\ga}2}\quad \mbox{for all} \ \ t\in ]0,t_0].
\eeq
In particular,
\beq\label{S1eq3a}
\liminf_{t\to 0^+}\f{{\rm rad}(u(t))}{\sqrt{t|\ln t|}}\geq \sqrt{2\ga-1}.
\eeq
Henceforth, we always denote ${\rm rad}(u(t))$ to be the space analyticity radius of $u(t).$
}
\end{thm}

\begin{rmk}\label{S1rmka}
Herbst and Skibsted  asked the questions below (see page 194 of  \cite{HS}):
\begin{enumerate}
\item[(i)] Are the bounds \eqref{S1eq12} and \eqref{S1eq3a} optimal for $\ga\in ]1/2,3/2[?$

 \item[(ii)] Are there better bounds than those deducible from Theorem \ref{thHS} if $\ga>\f32?$

 \item[(iii)] Can  the asymptotic
 \beq\label{S1eq11}
\lim_{t\to 0^+}\f{{\rm rad}(u(t))}{\sqrt{t}}=\infty.
\eeq
 be improved for the critical case $\ga=\f12$?
 \end{enumerate}
 \end{rmk}

 The purpose of this paper is { to settle the questions}  in Remark \ref{S1rmka} proposed by Herbst and Skibsted.
Our first main result {addresses the subcritical case $H^{\gamma}$ with $\gamma>\frac 12$}.

\begin{thm}\label{thm1}
{\sl Let $u_0 \in H^{\gamma}$ with $\gamma>\frac 12$ and be divergence-free. There exists $T>0$ so that the system
$(GNS)$ has a unique solution $u\in C([0,T]; H^\ga)\cap L^2(]0,T[; \dot H^{\ga+1}).$ Furthermore,
there exists $t_0\leq T$ so that for any sufficiently small $\d>0$  with $\ga>\f12+2\d,$ there holds
\beq\label{S2eq35}
\bigl\|e^{{\la(t) \sqrt{t}}|D|}u(t)\bigr\|_{\dot H^{\f12+\d}}
\le C t^{-\f12\left(\ga+\d-\f12\right)}|{\ln t}|^{\f12\left(\ga-\f12\right)}
e^{\f34\beta(t)}\quad \mbox{for all} \ \ t\in ]0,t_0],
\eeq
where
\beq\label{S1eq2}
\begin{split}
&\lambda(t) \eqdefa \sqrt{  (2\gamma-1 )\bigl( |{\ln t}| + \ln |{\ln t}|\bigr) +3 \beta(t) } \with\\
&
{\eta_J^\ga(t)}\eqdefa \max_{0\le \tau \le {t} } \| 1_{|\xi| \ge 0.01 J} |\xi|^{\gamma}
\widehat{ u }(\tau,\xi)\|_{L_{\xi}^2}\andf\\
&\beta(t)\eqdefa
\begin{cases}
\min\bigl\{\  |\ln {\eta_{t^{-\frac 12}}^\ga(t)}  |, \, \frac 12(\gamma-\frac 12) |{\ln t}|\ \bigr\}, \qquad \text{if $\eta^\ga_{t^{-\frac 12}}(t)>0$};  \\
\frac 12(\gamma-\frac 12) |{\ln t}|, \qquad \text{if $\eta^\ga_{t^{-\frac 12}}(t)=0$}.
\end{cases}.
\end{split}
\eeq
 In particular, we have
\beq\label{S1eq14}
\f{{\rm rad}(u(t))}{\la(t) \sqrt{t}}\geq 1\quad \mbox{for all} \ \ t\in ]0,t_0].
\eeq
}
\end{thm}

\begin{rmk}\label{S1rmk3}
We observe from \eqref{S1eq2} and \eqref{S1eq14} that
\beno
\f{{\rm rad}(u(t))}{\sqrt{t\bigl(|\ln t|+\ln |\ln t|\bigr)}}\geq \sqrt{2\ga-1}\quad \mbox{for all} \ \ t\in ]0,t_0],
\eeno
which not only improves the analyticity radius of the solution to $(NS)$ obtained in \eqref{S1eq3a}, but also completes
the case in Theorem \ref{thHS} for any $\ga\geq \f32.$ We thus {completely
settle} the questions (i) and (ii) in Remark \ref{S1rmka} {proposed} by Herbst and Skibsted in \cite{HS}.
\end{rmk}

\begin{rmk} \label{S1rmk1}\begin{enumerate}
\item[(1)] For  any given
solution $u \in C([0,T]; H^{\gamma})$ of the system $(GNS)$ with
initial data $u_0 \in H^{\gamma}$, we shall prove in Lemma \ref{S2lem1} that ${\eta_J^\ga(t)} \to 0$ as $J\to \infty$.
We should point it out that the case $\eta^\ga_{t^{-\frac 12}}(t)=0$ in the definition of $\beta(t)$ is trivial in the following sense:
according to the definition of $\eta^{\gamma}_J(t)$, when $\eta^{\ga}_{t^{-\frac 12}}(t)=0,$
it holds  that
$\mathrm{supp}(\widehat u(\tau,\cdot) )  \subset \{\xi: \quad |\xi| \le 0.01t^{-\f12} \}$ for all $\tau\in ]0,t].$  In particular
the space analyticity radius of $u(\tau)$ is arbitrarily large for $\tau\in ]0,t].$ This scenario does not seem to be easily ruled out for the nonlinear
problem. For example if we consider two-dimensional Navier-Stokes in vorticity form, the solution
$\omega(t) =e^{t\Delta} \omega_0$ with $\omega_0$ being radial is an explicit solution to the nonlinear equation.
If one takes $\widehat{\omega_0}$ to be compactly supported, then clearly $\omega(t)$ also has the same compact
support in the frequency space.

\item[(2)] Roughly speaking, the definition of $\beta(t)$ given by \eqref{S1eq2} is to accommodate the situation when the initial data $u_0$ has
higher smoothness (say in $H^{m}$ with $m>\gamma>\frac 12$) whereas the working space is $H^{\gamma}$. Apparently
in the case $u_0 \in H^m$ with $m>\gamma>\frac 12$, we have ${\eta_J}^\ga(t) \lesssim J^{-(m-\gamma)}$ and
$|\log {\eta^\ga_{t^{-\frac 12}}(t)}| \ge \frac 12 (m-\gamma) |\ln t| -C$ ($C>0$ is a constant).
If $m$ is large, we clearly see an ``upgrade" of analyticity radius of the amount
$\frac 12 (\gamma-\frac 12) |\ln t|$ thanks to our definition of $\beta(t)$.

\item[(3)]
The cut-off $\frac 12(\gamma-\frac 12) |\ln t|$ is for the convenience of analysis only.  In principle it can be replaced by other
suitable $\mathcal O (|\ln t|)$ term but the corresponding running parameters (in our nonlinear analysis, see for example
the estimate of the low frequency piece \eqref{1.9aa}) will have to be adjusted accordingly.  In practice we tacitly assume
that the working space $H^{\gamma}$ ``saturates" the smoothness of $u_0$ so that $\eta_J^\ga(t)$ decays
suitably slowly as $J\to \infty$. For this reason we chose the working cut-off $
\frac 12(\gamma-\frac 12) |\ln t|$ in order to ease the presentation.
 We shall not dwell on this subtle technical issue here.
 \end{enumerate}
\end{rmk}



 { A fundamental insight leading to the proof of
 Theorem \ref{thm1} is that the high frequency part
of the solution to $(GNS)$ controls its space analyticity radius.}
In fact, we shall reformulate $(GNS)$ to the following form:
\beq \label{S2eq20}
u=e^{t\D}u_{0}+\cB(u,u).
\eeq
We shall use the classical iteration scheme to construct the approximate solution sequence $\{u_n\}_{n\in\N}$ of \eqref{S2eq20}.
We first prove that there exists a positive time $T$ so that  $\{u_n\}$  converges to the solution $u$ of $(GNS)$ in $L^\infty([0,T];H^\ga)\cap L^2(]0,T[; \dot H^{\ga+1}).$
Then we prove that there exists a positive time $t_0\leq T$ so that  $\|u_n\|_{X_{t_0}}$  is uniformly bounded,  where
 the working norm $\|\cdot\|_{X_T}$ is {judiciously chosen as}
\beq\label{S1eq3}
\|u \|_{X_T}\eqdefa \bigl\| t^{\frac {\delta } 2} | \xi |^{\delta+\frac 12}  1_{|\xi| \ge   0.01\la T^{-\frac 12} }
{e^{-\frac {\lambda^2 t}{4T}+\lambda \frac t {\sqrt T} |\xi|}   \widehat u (t,\xi)}
\bigr\|_{L_T^{\infty}(L_{\xi}^2)}.
\eeq
Henceforth, $\d$ is a positive constant satisfying $\ga>\f12+2\d.$
Finally we prove the convergence of the approximate solution sequence $\{u_n\}_{n\in\N}$ in the norm: $$ \bigl\| t^{\frac {\delta } 2} | \xi |^{\delta+\frac 12}
{e^{-\frac {\lambda^2 t}{4T}+\lambda \frac t {\sqrt T} |\xi|}   \widehat u (t,\xi)}
\bigr\|_{L_T^{\infty}(L_{\xi}^2)}$$
for $\la=\la(T)$ with $\la(T)$ being given by \eqref{S1eq2}.


To answer the question (iii) of Remark \ref{S1rmka} for the critical case $\ga=\f12,$ we have the following result:

\begin{thm}\label{thm2}
{\sl Let $u_0\in \dot H^{\f12}$ and be divergence free. The system $(GNS)$ with initial data $u_0$
 has a unique solution $u\in C([0,T]; \dot H^{\f12})\cap L^2(]0,T[; \dot H^{\f32})$ for some positive
 time $T.$
 We denote \beq  \label{S1eq9}
 \begin{split}
\zeta_J^{\ga}(t)\eqdefa \max_{0\le \tau \le {t} } \| 1_{|\xi| \ge J} |\xi|^{\ga}
\widehat{ u }(\tau,\xi)\|_{L_{\xi}^2} \andf
 \lambda(t) \eqdefa \sqrt{ 3\min\bigl\{ |\ln {\zeta^{\f12}_{t^{-\frac 14}}(t)}|, \, |\ln t| \bigr\} },
\end{split}\eeq
 where $|\log {\zeta^{\f12}_{t^{-\frac 14}}}(t)|$ is tacitly
defined as $\infty$ if ${\zeta^{\f12}_{t^{-\frac 14}}}(t)=0$.
Then there exists a positive time $t_1\le T$ so that for all $ t\in ]0,t_1],$ there holds
\beq \label{S1eq10q}
\begin{split}
\bigl\|e^{\la(t)\sqrt{t}|D|}u(t)\|_{\dot H^{\f12+\d}}\leq & C(\|u_0\|_{\dot H^{\f12}})\Bigl(t^{-\f\d4} +t^{-\f\d2}
e^{\left(\f14+10^{-4}\right)\la^2(t)}{\zeta^{\f12}_{t^{-\frac 14}}}(t)\Bigr).
\end{split}
\eeq
}\end{thm}

\begin{rmk} \begin{enumerate}

\item[(1)] It is easy to observe from \eqref{S1eq10q} that
\beq \label{S1eq16} {{\rm rad}(u(t))}\geq \la(t)\sqrt{t}\quad \mbox{for all} \ \ t\in ]0,t_1]. \eeq
In Proposition \ref{S3prop1} we show that
$
\zeta_J^{\f12}(t) \to 0$ as $J\to\infty$. This implies $\lambda(t) \to \infty$ as $t\to 0^+$.
By \eqref{S1eq9} and \eqref{S1eq10q}, we deduce \eqref{S1eq11}, i.e.
\begin{align}
\lim_{t\to 0^+} \frac{\mathrm{rad}(u(t) )} {\sqrt{t} } =\infty.
\end{align}
In this sense the point-wise-in-time bound \eqref{S1eq16} offers a minuscule yet nontrivial improvement.
Note that \eqref{S1eq16} also gives an ``$\epsilon$"-improvement of \cite{CGZ3}.

\item[(2)] In \cite{CGZ3}, the authors proved that for any global solution $u\in C([0,\infty[; \dot H^{\f12}(\R^3))$ of $(NS),$
there holds
\beno
\lim_{t\to \infty}\f{{\rm rad}(u(t))}{\sqrt{t}}=\infty. \eeno
We expect that similar result as \eqref{S1eq16} should be true for any global solution $u$ of $(NS)$ with time $t$ being large enough. However we shall not pursue this {interesting} direction here.
\end{enumerate}
\end{rmk}

Let us end this section with some notations that we shall use throughout this paper.\\

\noindent{\bf Notations:}
\begin{itemize}
\item We denote $C$ to be an absolute constant
whose value may vary from line to line.  For any two positive quantities $X$ and $Y$, we write
$X \lesssim_{Z_1,\cdots, Z_k} Y$ if  $X \le C_1 Y$ for some positive constant $C_1$ depending
on $(Z_1,\cdots, Z_k)$. We write $X \lesssim Y$ if $X\le C_2 Y$ for some harmless constant
$C_2>0$.  {Occasionally we use the notation $X\ll Y$ or $Y\gg X$ to denote $X\le c Y$, where
$c>0$ is a sufficiently small constant.}

\item We adopt the following convention for Fourier transform. For Schwartz function $a=a(x): \, \mathbb R^3 \to \mathbb C$,
we denote the Fourier transform
\begin{align} \notag
(\mathcal F a)(\xi) =\widehat{a}(\xi) \eqdefa 
\int_{\mathbb R^3} a(x) e^{-i x\cdot \xi} dx.
\end{align}
For  Schwartz function $b=b(\xi): \, \mathbb R^3 \to \mathbb C$, we denote the inverse Fourier transform
\begin{align} \notag
(\mathcal F^{-1} b)(x)  \eqdefa 
(2\pi)^{-3}
 \int_{\mathbb R^3} b(\xi) e^{i x \cdot \xi} d\xi.
\end{align}
The action of Fourier transform and inverse Fourier transform on tempered distributions can be defined
accordingly.

\item  We { use the notation  $t\to 0^+$ } to denote $t\to 0$ with $t>0$.

\item  We shall use the Japanese bracket notation $\langle x \rangle = (1+|x|^2)^{\frac 12}$ for $x\in \mathbb R^3$.
For $s\in \mathbb R$, we denote the smoothed fractional Laplacian
$\langle D \rangle^s =(I-\Delta)^{s/2}$ which corresponds to the Fourier multiplier
$(1+|\xi|^2)^{s/2}$. We also use $|D|^s=(-\Delta)^{s/2}$ to denote the fractional
Laplacian which corresponds to the symbol $|\xi|^s$. { We denote by $(f, g)_{\dot H^s}$ the usual
$\dot H^s$ inner product, namely
\begin{align}
(f,g)_{\dot H^s} = \int_{\mathbb R^3} |D|^s f |D|^s \bar{g}\, dx.
\end{align}
}

\item For vector-valued Schwartz function $u=(u_1,u_2,u_3): \mathbb R^3 \to \mathbb C^3$, we denote
\begin{align}
\| u\|_p =\| (|u_1|^2+ |u_2|^2+|u_3|^2)^{\frac 12} \|_{L^p(\mathbb R^3)},
\end{align}
where $L^p$ is the usual Lebesgue $L^p$-norm. The vector-valued Sobolev norm $H^s$ is similarly
defined. In yet other words we shall suppress the notational dependence of the vector-valued spaces.
For example we write $L^p(\mathbb R^3)^3$ simply as $L^p(\mathbb R^3)$.

\item We use $*$ to denote the convolution of two functions, namely for Schwartz functions
$f_1: \mathbb R^3 \to \mathbb C$
and $f_2: \mathbb R^3 \to \mathbb C$,
\begin{align} \notag
(f_1*f_2)(x) \eqdefa \int_{\mathbb R^3} f_1(x-y) f_2 (y)dy.
\end{align}

\item For a nonempty set $A$, we use $1_A$ to denote the usual indicator function, i.e.
\begin{align}
1_A \eqdefa \begin{cases}
1, \qquad \text{if $x\in A$}; \\
0, \qquad \text{otherwise}.
\end{cases}
\end{align}
For example in Section 2, we have
\begin{align}
1_{|\xi| \ge 0.01 N_1} =
\begin{cases}
1, \qquad \text{if $|\xi| \ge 0.01 N_1$}; \\
0, \qquad \text{otherwise}.
\end{cases}
\end{align}

\item For two vectors $u=(u_1,u_2,u_3) \in \mathbb R^3$ and $v=(v_1,v_2,v_3) \in \mathbb R^3$, we employ
the usual tensor notation
\begin{align}
(u\otimes v)_{ij} \eqdefa u_i v_j.
\end{align}

\item
For a Banach space $B$, we shall use the shorthand $L^p_T(B)$ for $\bigl\|\|\cdot\|_B\bigr\|_{L^p(0,T;dt)}$.
We denote by $C([0, T]; B)$ the Banach space of continuous functions from $[0, T]$ to $B$ endowed with
the norm
\begin{align}
\| u\|_{C([0,T]; B)}\eqdefa\sup_{0\le t \le T} \| u(t ) \|_B.
\end{align}
\end{itemize}

\vskip 0.2cm

\setcounter{equation}{0}

\section{The subcritical case: $\ga>\f12$}

The goal of this section is to present the proof of Theorem \ref{thm1}.
Henceforth, for $\la,T>0$ and $a\in \cS'(\R^3),$ we always denote $N_1\eqdefa \lambda T^{-\frac 12}$ and $\widehat{a}(t,\xi)\eqdefa\cF_{x\to\xi}(a)(t,\xi)$,
the Fourier transform of $a$ with respect to the space variable, and  we
decompose $a$ as
\beq\label{S2eq1}
a=a_{\rm l}+a_{\rm h}\with \widehat{a_{\rm l}}(t,\xi)\eqdefa\widehat{a}(t,\xi)\cdot 1_{|\xi| <0.01  N_1} \andf \widehat{a_{\rm h}} \eqdefa \widehat{a}(t,\xi) \cdot 1_{|\xi|\ge 0.01 N_1 }.
\eeq
{We  denote}
\beq\label{S2eq2}
\begin{split}
&\qquad\widehat{\T(a)}(t,\xi)\eqdefa e^{-\frac {\lambda^2 t}{4  T}+ \lambda \frac t {\sqrt T} |\xi|} \widehat{a}(t,\xi).
\end{split}
\eeq

Let us first deal with the low frequency part of $a.$ {Throughout this paper we shall tacitly assume
$T\ll 1$ since $T$ will be eventually taken sufficiently small. Since $\lambda$ will also eventually be taken
sufficiently large ($\lambda=\mathcal O(\sqrt{|\ln T|})$ in the main order), we shall also tacitly assume $\lambda \gg 1$ to avoid
any pathologies in the computation.
For example in Lemma \ref{S2lem2} below, we have $(\lambda \sqrt T)^{-1} <0.01N_1=0.01 \lambda T^{-\frac 12}$.}

\begin{lem}\label{S2lem2}
{\sl Let $\ga>\f12+\d$ with $\d>0$ and $a\in L^\infty([0,T]; H^\ga).$ Then for any $t\leq T,$ one has
\beq  \label{1.9aa}
 \| \T(a_{\rm l})(t) \|_{\dot H^{\f12+\d}}
 \lesssim \bigl(1+\la^{\f12+\d+\ga}T^{\f12\left(\ga-\f12-\d\right)}e^{0.01\la^2}\bigr) e^{-\frac {\lambda^2 t}{4T} } \| a \|_{L^\infty_T(H^{\gamma})}.
\eeq}
 \end{lem}

\begin{proof}
We observe that for  $\ga>\f12+\d,$
\begin{align*}
 |\xi|^{\frac 12+\delta} \frac {e^{\lambda  \frac t {\sqrt T} |\xi|} } {\langle \xi \rangle^{\gamma}}
1_{|\xi| <0.01 N_1}
\leq & |\xi|^{\frac 12+\delta} \frac {e^{\lambda  \frac t {\sqrt T} |\xi|} } {\langle \xi \rangle^{\gamma}}
1_{|\xi| \le \left(\la\sqrt{T}\right)^{-1} }\\
& +|\xi|^{\frac 12+\delta} \frac {e^{\lambda  \frac t {\sqrt T} |\xi|} } {\langle \xi \rangle^{\gamma}}
\cdot 1_{\left(\la\sqrt{T}\right)^{-1}<|\xi| <0.01 N_1 }\\
\lesssim & \; 1+\la^{\f12+\d+\ga}T^{\f12\left(\ga-\f12-\d\right)}e^{0.01\la^2}.
\end{align*}
{By \eqref{S2eq2}}, we infer
\begin{align*}
 \| |D|^{\delta+\frac 12} \T(a_{\rm l})(t) \|_{L^2}
 \lesssim & \| |\xi|^{\delta+\frac 12} e^{-\frac {\lambda^2 t} {4T}+\lambda \frac t {\sqrt T } |\xi|} 1_{|\xi| < 0.01
 N_1}
 \widehat a (t, \xi) \|_{L^2_\xi}\\
 \lesssim &\bigl(1+\la^{\f12+\d+\ga}T^{\f12\left(\ga-\f12-\d\right)}e^{0.01\la^2}\bigr) e^{-\frac {\lambda^2 t}{4T} } \|a\|_{L^\infty_T(H^{\gamma})}.
 \end{align*}
 \end{proof}

We remark  that there is a saving of $e^{-\frac {\lambda^2 t }{4T} }$ in the estimate \eqref{1.9aa} which will be used in the nonlinear estimates later.

Note from \eqref{S1eq3} and \eqref{S2eq2} that the  the $X_T$ semi-norm of $a$ is just  $\|t^{\frac {\delta} 2} |D|^{\delta+\frac 12} \T(a_{\rm h})(t) \|_{L_T^{\infty}(L^2)}$.
To estimate $\|a\|_{X_T}$, we first observe that for any $t\leq T\leq 1,$
\beq\label{S2eq3}
\begin{split}
  & \| 1_{0.01N_1 \le |\xi| \le 0.1 N_1 }  t^{\frac {\delta } 2}
  | \xi |^{\delta+\frac 12} e^{
 -\frac {\lambda^2 t} {4T} +\lambda \frac t {\sqrt T} |\xi| } \widehat a(t,\cdot) \|_{L^2}  \\
 &\lesssim  T^{\frac {\delta}2} (N_1 )^{-\left(\gamma-\frac 12-\delta\right)} e^{-\left(\f14-0.1\right)\f{\la^2t}T}
 \|  a_{\rm h} \|_{L_T^{\infty}(H^{\gamma})}  \\
 &\lesssim  T^{\frac 12(\gamma-\frac 12)}  \lambda^{-\left(\gamma-\frac 12-\delta\right)} \| a_{\rm h}\|_{L_T^{\infty}(H^{\gamma})}.
 \end{split}\eeq

Therefore
{ to complete the estimate of $\|u\|_{X_T}$ for any solution $u$ of $(GNS)$ in $L^\infty([0,T]; H^\ga),$} it remains for us to handle the  { main piece}
\beq\label{S2eq3a}
\| 1_{|\xi| \ge 0.1 N_1 }  t^{\frac {\delta } 2}
  | \xi |^{\delta+\frac 12} e^{
 -\frac {\lambda^2 t} {4T} +\lambda \frac t {\sqrt T} |\xi| } \widehat u(t,\xi) \|_{L^\infty_T(L^2_\xi)}.\eeq
 For this, we {appeal to} the following integral reformulation of $(GNS)$:
 $$
 u=e^{t\D}u_0+\int_0^te^{(t-s)\D}{Q(u,u)}(s)\,ds, $$ with the bilinear form $Q(f,g)$ being given by \eqref{S1eq1}.
 {The avid reader should think of $Q(f,g) \approx \mathcal R \partial (f g)$, where $\mathcal R$ is
 Riesz-type transform.}
{On the Fourier side we  need to   estimate}
$$1_{|\xi| \ge 0.1 N_1 }\int_0^te^{-(t-s)|\xi|^2}\widehat{Q(f,g)}(s,\xi)\,ds.$$
 Thanks to the { high} frequency cut-off $1_{|\xi|\ge 0.1N_1}$, there will be no low-low interactions of $f$ and $g$ in the nonlinear estimate. {Our main technical result is stated in the next
 proposition.  This is the most crucial ingredient used in the proof of Theorem \ref{thm1}.}

\begin{prop}\label{S2prop1}
{\sl Let  $Q(f,g)$ be the bilinear form given by \eqref{S1eq1}. Then for $\ga>\f12+2\d$ and $\eta_0$ being a small enough positive constant, one has
\beq\label{S2eq13}
\begin{split}
  \bigl\| t^{\frac {\delta } 2}&1_{|\xi| \ge 0.1 N_1 }
  | \xi |^{\delta+\frac 12} e^{\lambda \frac t {\sqrt T } |\xi| -\frac {\lambda^2 t}{4 T} } \int_0^te^{-(t-s)|\xi|^2}\widehat{Q(f,g)}(s,\xi)\,ds \bigr\|_{L^\infty_T(L^2_\xi)}\\
  \lesssim &  \bigl( e^{4\eta_0\la^2} +\la^{-\d} \bigr)\la^{-{(\gamma-\frac 12+\delta) }}T^{\frac 12\left(\gamma-\frac 12+2\delta\right) } \\
&\ \times\bigl(\|f\|_{L_T^{\infty}(H^{\gamma})}\|g_{\rm h}\|_{L_T^{\infty}(H^{\gamma})}+\|f_{\rm h}\|_{L_T^{\infty}(H^{\gamma})}\|g\|_{L_T^{\infty}(H^{\gamma})}
\bigr)+ \lambda^{-\delta} e^{\frac {\lambda^2 } 4} \|f \|_{X_T}\|g\|_{X_T}\\
&
+ \lambda^{2-\d}T^{\frac {\delta}2}\bigl(1+\la^{\f12+\d+\ga}T^{\f12\left(\ga-\f12-\d\right)}e^{0.01\la^2}\bigr)\bigl(\|g\|_{L^\infty_T(H^{\ga})} \|f \|_{X_T}+\|f\|_{L^\infty_T(H^{\ga})} \|g \|_{X_T}\bigr),
   \end{split}
 \eeq where $f_{\rm h}$ and $g_{\rm h}$ are given by \eqref{S2eq1}.}
\end{prop}

\begin{proof}
For any $\eta_0>0,$ which will be taken sufficiently small later,   we split the integral $\int_0^t = \int_0^{\eta_0 t} + \int_{\eta_0 t}^t$
 and shall estimate
each piece separately.

In view of \eqref{S1eq1},  we decompose the piece $\int_0^{\eta_0 t}$ into the following two parts:
\beq\label{S2eq4}
\begin{split}
t^{\frac {\delta} 2} &\bigl\|  |\xi|^{\frac 12+\d}  1_{|\xi| \ge 0.1 N_1 }
 e^{\lambda \frac t {\sqrt T } |\xi| -\frac {\lambda^2 t}{4 T} }\int_0^{\eta_0 t}   e^{-(t-s) |\xi|^2}
\widehat{Q(f,g)}(s,\xi)\, ds \bigr\|_{L^2_\xi} \\
=&t^{\frac {\delta} 2}\bigl\|  |\xi|^{\frac 12+\d}  1_{|\xi| \ge 0.1 N_1 }
\int_0^{\eta_0 t}  e^{-t\left(|\xi|-\f\la{2\sqrt{T}}\right)^2} e^{s|\xi|^2}
\widehat{Q(f,g)}(s,\xi)\, ds \bigr\|_{L^2_\xi} \\
 \lesssim&    \bigl\| t^{\frac {\delta} 2}
 1_{|\xi| \le 2 N_1}
 \int_0^{\eta_0 t} |\xi|^{\frac 32+\delta}  e^{s |\xi|^2}
 1_{|\xi|\ge 0.1 N_1 } |\widehat{f}\ast\widehat{g} (s, \xi)| \,ds \bigr\|_{L^2_\xi } \\
 &+ \bigl\| t^{\frac {\delta} 2}   1_{|\xi| \ge 2 N_1 }
 \int_0^{\eta_0 t} |\xi|^{\frac 32+\delta}  e^{- \frac{t}{10} |\xi|^2}
|\widehat{f}\ast\widehat{g} (s, \xi)| \, ds \bigr\|_{L^2_\xi},
 \end{split}\eeq
 provided that $\eta_0$ is small enough. {Here the smallness of $\eta_0$ is needed  for the second
 piece so that $e^{s|\xi|^2} e^{-t\cdot \frac {9}{16} |\xi|^2} \le e^{-\frac t{10} |\xi|^2}$.
 In particular $\eta_0<0.1$ suffices here.}

 By frequency localization and \eqref{S2eq1}, we have
 \beq\label{S2eq5}
 1_{|\xi| \ge 0.1 N_1 }\widehat{f}\ast\widehat{g} = 1_{|\xi|\ge 0.1 N_1 }
 \bigl(\widehat{f}_{\rm l} \ast \widehat{g}_{\rm h}  + \widehat{f}_{\rm h} \ast \widehat{g}_{\rm l}
 +
 \widehat{f}_{\rm h} \ast \widehat{g}_{\rm h}\bigr).
 \eeq
Then we get, by applying Young's inequality, $\|\wh{f}\|_{L^{p'}}\leq C\|f\|_{L^p}$ for $p\in [1,2]$ and $p'\eqdefa\f{p}{p-1},$ that
\begin{align*}
 &  \int_{0}^{\eta_0t}   \bigl\|  1_{|\xi| \le 2N_1}   |\xi|^{\frac 32 +\delta}e^{s|\xi|^2} |1_{|\xi|\ge 0.1 N_1} |\widehat{f}\ast\widehat{g}(s,\xi) | \bigr\|_{L_{\xi}^2} ds \\
 &\lesssim \int_{0}^{\eta_0 t} e^{4 s N_1^2} N_1^{\frac 32+\delta}\bigl\|1_{|\xi| \le 2N_1}\bigr\|_{L^{\f6{1-4\d}}_\xi}\bigl\| 1_{|\xi|\ge 0.1 N_1} \widehat{f}\ast\widehat{g}(s,\cdot)\bigr\|_{L^{\f3{1+2\d}}_\xi}\,ds\\
& \lesssim  \int_{0}^{\eta_0 t} e^{4 s N_1^2} N_1^{\frac 32+\delta}
  \cdot N_1^{\frac 12-2\delta}  \bigl(\| f_{\rm l} g_{\rm h}(s)\|_{L^{\frac 3{2(1-\delta)}}} +\| f_{\rm h} g_{\rm l}(s)\|_{L^{\frac 3{2(1-\delta)}}} +\|f_{\rm h} g_{\rm h}(s) \|_{L^{\frac 3{2(1-\delta) }}}\bigr) \,ds,
   \end{align*}
 from which, Sobolev imbedding inequality and \eqref{S2eq1}, we infer
  \begin{align*}
 &  \int_{0}^{\eta_0t}   \|  1_{|\xi| \le 2N_1}   |\xi|^{\frac 32 +\delta}e^{s|\xi|^2} |1_{|\xi|\ge 0.1 N_1} |\widehat{f}\ast\widehat{g}(s,\xi) | \|_{L_{\xi}^2} ds \\
 & \lesssim \int_0^{\eta_0 t} e^{4s N_1^2} N_1^{2-\delta}
  \bigl( \| |D|^{\frac 12+2\delta} f_{\rm l} \|_{L^2} \| |D|^{\frac 12} g_{\rm h} \|_{L^2}
 +  \| |D|^{\frac 12} f_{\rm h} \|_{L^2}\| |D|^{\frac 12+2\delta} g_{\rm l} \|_{L^2}\\
 &\qquad+
 \| |D|^{\frac 12+ \delta} f_{\rm h} \|_{L^2} \| |D|^{\frac 12+ \delta} g_{\rm h} \|_{L^2} \bigr)\, ds \\
& \lesssim \int_0^{\eta_0 t} e^{4s N_1^2} N_1^{2-\delta} \,ds \Bigl( N_1^{- \gamma+\frac 12}\bigl(\|f\|_{L_t^{\infty}(H^{\gamma})}\|g_{\rm h}\|_{L_t^{\infty}(H^{\gamma})}+\|f_{\rm h}\|_{L_t^{\infty}(H^{\gamma})}\|g\|_{L_t^{\infty}(H^{\gamma})}\bigr)\\
& + N_1^{-2(\gamma-\frac 12-\delta)} \|f_{\rm h}\|_{L_t^{\infty}(H^{\gamma})} \|g_{\rm h}\|_{L_t^{\infty}(H^{\gamma})}\Bigr),
\end{align*}
due to  $\ga>\f12+2\d$, we achieve
\begin{align*}
 &  \int_{0}^{\eta_0t}   \|  1_{|\xi| \le 2N_1}   |\xi|^{\frac 32 +\delta}e^{s|\xi|^2} |1_{|\xi|\ge 0.1 N_1} |\widehat{f}\ast\widehat{g}(s,\xi) | \|_{L_{\xi}^2} ds \\
 &\lesssim  N_1^{-\delta} e^{4\eta_0 T N_1^2}  N_1^{-\gamma+\frac 12} \bigl(\|f\|_{L_t^{\infty}(H^{\gamma})}\|g_{\rm h}\|_{L_t^{\infty}(H^{\gamma})}+\|f_{\rm h}\|_{L_t^{\infty}(H^{\gamma})}\|g\|_{L_t^{\infty}(H^{\gamma})}
\bigr)\\
& \lesssim   \bigl(\la^{-1}T^{\frac 12}\bigr)^{\gamma-\frac 12+\delta } e^{4\eta_0\la^2}\bigl(\|f\|_{L_t^{\infty}(H^{\gamma})}\|g_{\rm h}\|_{L_t^{\infty}(H^{\gamma})}+\|f_{\rm h}\|_{L_t^{\infty}(H^{\gamma})}\|g\|_{L_t^{\infty}(H^{\gamma})}
\bigr).
   \end{align*}

Along the same line,  we deduce that
 \begin{align*}
 & t^{\frac {\delta} 2}  \bigl\| 1_{|\xi| \ge 2 N_1 }
 \int_0^{\eta_0 t} |\xi|^{\frac 32+\delta}  e^{- \frac{t}{10} |\xi|^2}
\widehat{f}\ast\widehat{g} (s, \xi) \, ds \bigr\|_{L^2} \\
&\lesssim   t^{-1}\bigl\||\xi|^{-\f12} 1_{|\xi| \ge 2 N_1 }  \bigr\|_{L^{\f6{1-4\d}}}\int_0^{\eta_0 t}\| 1_{|\xi|\ge 0.1 N_1} \widehat{f}\ast\widehat{g}(s,\cdot)\|_{L^{\f3{1+2\d}}}\,ds\\
&\lesssim  t^{-1}N_1^{-2\d}  \int_0^{\eta_0 t}  \bigl(\| f_{\rm l} g_{\rm h}(s)\|_{L^{\frac 3{2(1-\delta)}}} +\| f_{\rm h} g_{\rm l}(s)\|_{L^{\frac 3{2(1-\delta)}}} +\|f_{\rm h} g_{\rm h}(s) \|_{L^{\frac 3{2(1-\delta) }}}\bigr)
 \,ds\\
 &\lesssim \bigl(\lambda^{-1} T^{\frac 12}\bigr)^{\gamma-\frac 12+2\d}\bigl(\|f\|_{L_t^{\infty}(H^{\gamma})}\|g_{\rm h}\|_{L_t^{\infty}(H^{\gamma})}+\|f_{\rm h}\|_{L_t^{\infty}(H^{\gamma})}\|g\|_{L_t^{\infty}(H^{\gamma})}
\bigr).
  \end{align*}

{Plugging} the above estimates into \eqref{S2eq4}, we {obtain}
  \beq \label{S2eq7}
  \begin{split}
 &\bigl\| t^{\frac {\delta} 2} |\xi|^{\frac 12+\d}  1_{|\xi| \ge 0.1 N_1 }
  e^{\lambda \frac t {\sqrt T } |\xi| -\frac {\lambda^2 t}{4 T} }
 \int_0^{\eta_0 t}   e^{-(t-s) |\xi|^2}
\widehat{Q(f,g)}(s,\xi)\, ds \bigr\|_{L^\infty_T(L^2_\xi)} \\
&\lesssim \bigl( e^{4\eta_0\la^2} +\la^{-\d} \bigr)\la^{-{(\gamma-\frac 12+\delta) }}T^{\frac 12\left(\gamma-\frac 12+2\delta\right) } \\
&\qquad\times\bigl(\|f\|_{L_T^{\infty}(H^{\gamma})}\|g_{\rm h}\|_{L_T^{\infty}(H^{\gamma})}+\|f_{\rm h}\|_{L_T^{\infty}(H^{\gamma})}\|g\|_{L_T^{\infty}(H^{\gamma})}
\bigr).
\end{split}
\eeq

To deal with the other piece $\int_{\eta_0 t}^t,$ we need the following lemma,  the proof of which will be postponed
after we finish the proof of Proposition \ref{S2prop1}.

\begin{lem} \label{S2lem1}
{\sl Let $N_0\eqdefa\frac {N_1}2 \gg 1$.
Then for all $0<t\le T$, one has
\beq\label{S2eq8}
\bigl\|t^{\frac {\delta}2} \int_{\eta_0 t}^t 1_{|\xi| \le 2N_0} e^{ N_0^2s}  |\xi|^{\frac 32 +\delta}
\widehat{F}(s, \xi)  \,ds \|_{L^\infty_T(L_{\xi}^2)}
\lesssim \lambda^{-\delta} e^{ N_0^2T}\cdot \sup_{0<s\le T}\bigl( s^{\delta }
 \|F(s)\|_{L^{\frac 3 {2(1-\delta)}}}\bigr),\eeq
 and 
 \beq\label{S2eq9}
\begin{split}
 \bigl\|t^{\frac {\delta} 2}
 \int_{\eta_0 t}^{t}1_{|\xi| \ge  2N_0} e^{ N_0^2 s} |\xi|^{\frac 32+\delta} & e^{-   \frac 1{10} (t-s) |\xi|^2}
 \widehat{F}(s,\xi) \,ds \|_{L^\infty_T(L_{\xi}^2)}\\
  &\qquad \qquad\lesssim
 \lambda^{-\delta}
 e^{ N_0^2T}\cdot \sup_{0<s\le T} \bigl(s^{\delta } \| F(s)\|_{L^{\frac 3 {2(1-\delta)}}}\bigr).
\end{split}\eeq}
\end{lem}

We now continue our estimate {of} the piece $\int_{\eta_0 t}^t.$
In view of \eqref{S2eq2}, we write
\beq\label{S2eq10}
 \begin{split}
 t^{\frac {\delta} 2} &\bigl\|  |\xi|^{\frac 12+\d}  1_{|\xi| \ge 0.1 N_1 }
 e^{\lambda \frac t {\sqrt T } |\xi| -\frac {\lambda^2 t}{4 T} }\int_{\eta_0 t}^t   e^{-(t-s) |\xi|^2}
\widehat{Q(f,g)}(s,\xi)\, ds \bigr\|_{L^2_\xi} \\
\lesssim &t^{\frac {\delta} 2} \bigl\| |\xi|^{\frac 12+\d}  1_{|\xi| \ge 0.1 N_1 }
 \int_{\eta_0 t}^{t} |\xi|  e^{-(t-s) |\xi|^2}
 e^{\lambda \frac t {\sqrt T } |\xi| -\frac {\lambda^2 t}{4 T} }
 |\widehat{f}\ast \widehat{g}(s,\xi)|\, ds \bigr\|_{L^2_\xi}  \\
 \lesssim&
   t^{\frac {\delta} 2} \bigl\||\xi|^{\f32+\d}
 \int_{\eta_0 t}^{t}   e^{- (t-s) |\xi|^2}
 e^{\lambda \frac {t-s} {\sqrt T} |\xi| +\frac {\lambda^2 s} {2 T}-\frac {\lambda^2 t}{4T} }
 1_{|\xi| \ge 0.1 N_1 }|\widehat{\T(f)}\ast \widehat{\T(g)}(s,\xi)|\, ds \bigr\|_{L^2_\xi}.
 \end{split}\eeq
 By frequency localization \eqref{S2eq5}, it amounts to handle the estimates related to the terms: $\widehat{T(f_{\rm l})}\ast \widehat{\T(g_{\rm h})},$
 $\widehat{\T(f_{\rm h})}\ast \widehat{\T(g_{\rm l})}$ and $\widehat{\T(f_{\rm h})}\ast \widehat{\T(g_{\rm h})}.$

 We first estimate the contribution due to $\widehat{\T(f_{\rm h})}\ast \widehat{\T(g_{\rm h})}.$ We further decompose
 the integrand into the high and low frequency parts so that
  \begin{align*}
  t^{\frac {\delta} 2}&\bigl\|  |\xi|^{\f32+\d}
 \int_{\eta_0 t}^{t}    e^{- (t-s) |\xi|^2}
 e^{\lambda \frac {t-s} {\sqrt T} |\xi| +\frac {\lambda^2 s} {2 T}-\frac {\lambda^2 t}{4T} }
 |\widehat{\T(f_{\rm h})}\ast \widehat{\T(g_{\rm h})}(s,\xi)|\, ds \bigr\|_{L^2_\xi} \\
 \lesssim & t^{\frac {\delta} 2} \bigl\|
 \int_{\eta_0 t}^{t} |\xi|^{\frac 32+\delta} e^{-   (t-s) \left(|\xi| - \frac {\lambda } {2\sqrt T } \right)^2  +\frac {\lambda^2 s}{4T}}
 |\widehat{\T(f_{\rm h})}\ast \widehat{\T(g_{\rm h})}(s,\xi)|\, ds \bigr\|_{L^2_\xi}
  \\
  \lesssim &t^{\frac {\delta} 2} \bigl\|
 \int_{\eta_0 t}^{t} |\xi|^{\frac 32+\delta} 1_{|\xi| \le N_1} e^{\frac {\lambda^2 s}{4T} } |\widehat{\T(f_{\rm h})}\ast \widehat{\T(g_{\rm h})}(s,\xi)|\, ds \bigr\|_{L^2_\xi}
   \\
  &+ t^{\frac {\delta} 2}\bigl\|
 \int_{\eta_0 t}^{t} |\xi|^{\frac 32+\delta} 1_{|\xi| \ge  N_1 }e^{-   \frac 1{10} (t-s) |\xi|^2
 +\frac {\lambda^2 s }{4T} }|\widehat{\T(f_{\rm h})}\ast \widehat{\T(g_{\rm h})}(s,\xi)|\, ds \bigr\|_{L^2_\xi},
 \end{align*}
 from which and Lemma \ref{S2lem1}, for $\la\geq 1$ and $T\leq 1,$ we infer
 \beq\label{S2eq11}
 \begin{split}
  t^{\frac {\delta} 2}&\bigl\|  |\xi|^{\f32+\d}
 \int_{\eta_0 t}^{t}    e^{- (t-s) |\xi|^2}
 e^{\lambda \frac {t-s} {\sqrt T} |\xi| +\frac {\lambda^2 s} {2 T}-\frac {\lambda^2 t}{4T} }
 |\widehat{\T(f_{\rm h})}\ast \widehat{\T(g_{\rm h})}(s,\xi)|\, ds \bigr\|_{L^2_\xi} \\
 \lesssim & \lambda^{-\delta} e^{\frac {\lambda^2 }4} \cdot \sup_{0<s \le T} \bigl(s^{\delta}
 \| \T(f_{\rm h}) \T(g_{\rm h})(s,\cdot) \|_{L^{\frac 3 {2(1-\delta)} } }\bigr)  \\
 \lesssim & \lambda^{-\delta} e^{\frac {\lambda^2 }4} \cdot  \sup_{0<s \le T}
 \bigl(s^{\delta} \| |D|^{\frac 12 +\delta} \T(f_{\rm h})(s,\cdot) \|_{L^2}  \| |D|^{\frac 12 +\delta} \T(g_{\rm h})(s,\cdot) \|_{L^2}\bigr) \\
 \lesssim &  \lambda^{-\delta} e^{\frac {\lambda^2 } 4} \|f \|_{X_T} \|g\|_{X_T}.
 \end{split}
 \eeq

Next we estimate the contribution due to $\widehat{\T(f_{\rm h})\T(g_{\rm l})}$. Indeed along the same line
{as} the estimate of
\eqref{S2eq11}, we write
\begin{align*}
  t^{\frac {\delta} 2}&\bigl\| |\xi|^{\f32+\d}
 \int_{\eta_0 t}^{t}    e^{- (t-s) |\xi|^2}
 e^{\lambda \frac {t-s} {\sqrt T} |\xi| +\frac {\lambda^2 s} {2 T}-\frac {\lambda^2 t}{4T} }
|\widehat{\T(f_{\rm h})}\ast \widehat{\T(g_{\rm l})}(s,\xi)|\, ds \bigr\|_{L^2_\xi}  \\
  \lesssim &t^{\frac {\delta} 2} \bigl\|
 \int_{\eta_0 t}^{t} |\xi|^{\frac 32+\delta} 1_{|\xi| \le N_1} e^{\frac {\lambda^2 s}{4T} }
 |\widehat{\T(f_{\rm h})}\ast \widehat{\T(g_{\rm l})}(s,\xi)|\, ds \bigr\|_{L^2_\xi} \\
  &+ t^{\frac {\delta} 2} \bigl\|
 \int_{\eta_0 t}^{t} |\xi|^{\frac 32+\delta} 1_{|\xi| \ge  N_1 }e^{-   \frac 1{10} (t-s) |\xi|^2
 +\frac {\lambda^2 s }{4T} } |\widehat{\T(f_{\rm h})}\ast \widehat{\T(g_{\rm l})}(s,\xi)|\, ds \bigr\|_{L^2_\xi}.
  \end{align*}
  Yet it follows from Young's inequality that
  \begin{align*}
  t^{\frac {\delta} 2}& \bigl\|
 \int_{\eta_0 t}^{t} |\xi|^{\frac 32+\delta} 1_{|\xi| \le N_1} e^{\frac {\lambda^2 s}{4T} }
 |\widehat{\T(f_{\rm h})}\ast \widehat{\T(g_{\rm l})}(s,\xi)|\, ds \bigr\|_{L^2_\xi} \\
 \lesssim& t^{\frac {\delta} 2}N_1^{\f32+\d}
 \int_{\eta_0 t}^{t}\bigl\| 1_{|\xi| \le N_1}\bigr\|_{L^{\f6{1-4\d}}}e^{\frac {\lambda^2 s}{4T} }
  \|\widehat{\T(f_{\rm h})}\ast \widehat{\T(g_{\rm l})}(s,\cdot)\|_{L^{\f3{1+2\d}}}\, ds\\
  \lesssim &N_1^{2-\d}T^{1-\f\d2}\sup_{0<s \le T}\Bigl( s^{\delta}e^{\frac {\lambda^2 s}{4T} }
  \|\T(f_{\rm h})\T(g_{\rm l})(s)\|_{L^{\frac 3 {2(1-\delta)} }}\Bigr),
  \end{align*}
  and 
  \begin{align*}
  t^{\frac {\delta} 2}& \bigl\|
 \int_{\eta_0 t}^{t} |\xi|^{\frac 32+\delta} 1_{|\xi| \ge  N_1 }e^{-   \frac 1{10} (t-s) |\xi|^2
 +\frac {\lambda^2 s }{4T} } |\widehat{\T(f_{\rm h})}\ast \widehat{\T(g_{\rm l})}(s,\xi)|\, ds \bigr\|_{L^2_\xi}\\
  \lesssim &t^{\frac {\delta} 2}\int_{\eta_0 t }^{t}(t-s)^{-1+\f\d2}s^{-\d}\,ds 
  \sup_{0<s \le T}\Bigl( s^{\delta}e^{\frac {\lambda^2 s}{4T} }
  \|\T(f_{\rm h})\T(g_{\rm l})(s)\|_{L^{\frac 3 {2(1-\delta)} }}\Bigr)\\
  \lesssim & \sup_{0<s \le T}\Bigl( s^{\delta}e^{\frac {\lambda^2 s}{4T} }
  \|\T(f_{\rm h})\T(g_{\rm l})(s)\|_{L^{\frac 3 {2(1-\delta)} }}\Bigr).
  \end{align*} 
 As a consequence, we deduce that
 \beq\label{S2eq11a}
\begin{split}
  t^{\frac {\delta} 2}&\bigl\| |\xi|^{\f32+\d}
 \int_{\eta_0 t}^{t}    e^{- (t-s) |\xi|^2}
 e^{\lambda \frac {t-s} {\sqrt T} |\xi| +\frac {\lambda^2 s} {2 T}-\frac {\lambda^2 t}{4T} }
|\widehat{\T(f_{\rm h})}\ast \widehat{\T(g_{\rm l})}(s,\xi)|\, ds \bigr\|_{L^2_\xi} \\
 \lesssim & (\lambda^{2-\delta} +1) \cdot \sup_{0<s \le T}\Bigl( s^{\delta}e^{\frac {\lambda^2 s}{4T} }
  \|\T(f_{\rm h})\T(g_{\rm l})(s)\|_{L^{\frac 3 {2(1-\delta)} }}\Bigr) \\
 \lesssim & \lambda^{2-\delta}  \cdot  \sup_{0<s \le T}
 \Bigl(s^{\delta}  e^{\frac {\lambda^2 s}{4T} } \| |D|^{\frac 12+\delta} \T(f_{\rm h})(s) \|_{L^2}
 \cdot \| |D|^{\frac 12 +\delta} \T(g_{\rm l})(s) \|_{L^2} \Bigr)\\
 \lesssim &   \lambda^{2-\delta} T^{\frac {\delta}2} \|f\|_{X_T} \sup_{0<s \le T}
  \bigl(e^{\frac {\lambda^2 s}{4T} }\|\T(g_{\rm l})(s)\|_{\dot H^{\f12+\d}}\bigr),
 \end{split}\eeq
 from which and \eqref{1.9aa}, we infer 
  $t\leq T,$
 \begin{align*}
 &\bigl\|t^{\frac {\delta} 2} |\xi|^{\f32+\d}
 \int_{\eta_0 t}^{t}    e^{- (t-s) |\xi|^2}
 e^{\lambda \frac {t-s} {\sqrt T} |\xi| +\frac {\lambda^2 s} {2 T}-\frac {\lambda^2 t}{4T} }
|\widehat{\T(f_{\rm h})}\ast \widehat{\T(g_{\rm l})}(s,\xi)|\, ds \bigr\|_{L^\infty_T(L^2_\xi)} \\
&\lesssim \lambda^{2-\delta} T^{\frac {\delta}2}\bigl(1+\la^{\f12+\d+\ga}T^{\f12\left(\ga-\f12-\d\right)}e^{0.01\la^2}\bigr)\|f\|_{X_T} \|g\|_{L^\infty_T(H^{\ga})}.
 \end{align*}
By substituting the above estimate and \eqref{S2eq11} into \eqref{S2eq10}, we obtain
\beq \label{S2eq12}
\begin{split}
 \bigl\| t^{\frac {\delta} 2}& |\xi|^{\frac 12+\d}  1_{|\xi| \ge 0.1 N_1 }
 e^{\lambda \frac t {\sqrt T } |\xi| -\frac {\lambda^2 t}{4 T} }\int_{\eta_0 t}^t   e^{-(t-s) |\xi|^2}
\widehat{Q(f,g)}(s,\xi)\, ds \bigr\|_{L^\infty_T(L^2_\xi)} \\
\lesssim &  \lambda^{-\delta} e^{\frac {\lambda^2 } 4}  \|f \|_{X_T}\|g\|_{X_T}+ \lambda^{2-\d}T^{\frac {\delta}2} \bigl(1+\la^{\f12+\d+\ga}T^{\f12\left(\ga-\f12-\d\right)}e^{0.01\la^2}\bigr)
\\
&\qquad\qquad\qquad\qquad\qquad\times
 \bigl(\|g\|_{L^\infty_T(H^{\ga})}  \|f \|_{X_T}+\|f\|_{L^\infty_T(H^{\ga})}  \|g\|_{X_T}\bigr).
 \end{split}\eeq

 By summarizing the estimates \eqref{S2eq7} and \eqref{S2eq12}, we obtain \eqref{S2eq13}.  This completes the proof of
 Proposition \ref{S2prop1}.
\end{proof}

Let us now present the proof of Lemma \ref{S2lem1}.

\begin{proof}[Proof of Lemma \ref{S2lem1}]
We first get, by applying Young's inequality, that
\begin{align*}
 &  \int_{\eta_0 t}^t e^{s N_0^2} \|  1_{|\xi| \le 2N_0}  |\xi|^{\frac 32 +\delta} \widehat{F}(s, \xi)  \|_{L_{\xi}^2} \,ds
  \\
 & \lesssim  \int_{\eta_0 t}^t e^{s N_0^2} N_0^{\frac 32+\delta}
  \bigl\| 1_{|\xi| \le 2N_0}\bigr\|_{L^{\f6{1-4\d}}} \|\widehat{F}(s, \xi) \|_{L^{\f3{1+2\d}}}\, ds  \\
 & \lesssim  \int_{\eta_0 t}^t e^{s N_0^2} N_0^{\frac 32+\delta}
  \cdot N_0^{\frac 12-2\delta}  \|F(s)\|_{L^{\frac 3 {2(1-\delta)}}}\, ds \\
 & \lesssim   \sup_{0<s\le T}\bigl( s^{\delta } \| F(s) \|_{L^{\frac 3 {2(1-\delta)}}}\bigr)
 \cdot t^{-\frac {\delta}2}
  (N_0 \sqrt t)^{-\delta} \cdot  \bigl( e^{ N_0^2 t} -1\bigr).
 \end{align*}
 {Note} that the function $f(x) = x^{-\delta} (e^{x^2} -1)$ is monotonically increasing in $x>0$. In particular,
 for $t\leq T,$ one has ({recall $2N_0=N_1=\lambda T^{-\frac 12}$})
 \begin{align*}
 f(N_0 \sqrt t) \le f(N_0 \sqrt T) \le  2^{\delta} \lambda^{-\delta} e^{N_0^2 T},
 \end{align*}
which leads to \eqref{S2eq8}.

On the other hand, we have
 \begin{align*}
 & t^{\frac {\delta} 2}\bigl\|
 \int_{\eta_0 t}^{t} e^{ N_0^2 s} |\xi|^{\frac 32+\delta} 1_{|\xi| \ge  2N_0}e^{-   \frac 1{10} (t-s) |\xi|^2}
 \widehat{F}(s,\xi) \,ds \|_{L^2_\xi}\\
&\lesssim  t^{\frac {\delta} 2}  \int_{\eta_0 t}^{ t} e^{N_0^2 s}
 (t-s)^{-1+\frac {\delta}2} 
 \| F(s,\cdot)\|_{L^{\frac 3{2(1-\delta) }}} \,ds \\
 &
 \lesssim  t^{\frac {\delta} 2} \int_{\eta_0 t}^t e^{N_0^2 s}  (t-s)^{-1+\frac {\delta}2} s^{-\delta}\, ds
 \sup_{0<s\leq T} \bigl(s^{\delta} \|F(s,\cdot)\|_{L^{\frac 3{2(1-\delta) }}}\bigr)\\
 &\lesssim  \int_{\eta_0}^1 e^{N_0^2 t \tau} (1-\tau)^{-1+\frac {\delta} 2} d\tau  \cdot
 \sup_{0<s\le T} \bigl(s^{\delta} \|F(s,\cdot)\|_{L^{\frac 3{2(1-\delta) }}}\bigr)
 \\
  &\lesssim  \langle N_0^2 T\rangle^{-\frac {\delta}2}  e^{N_0^2 T}\sup_{0<s\leq T} \bigl(s^{\delta} \|F(s,\cdot)\|_{L^{\frac 3{2(1-\delta) }}}\bigr),
 \end{align*}
 which ensures \eqref{S2eq9}. In the last step above, we used the elementary inequality
 (for $\Lambda\gg 1$):
 \begin{align}
 \int_{\eta_0}^1 e^{\Lambda \tau} (1-\tau)^{-1+\frac {\delta}2} d
 \tau = e^{\Lambda} \int_0^{1-\eta_0} e^{-\Lambda s} s^{-1+\frac {\delta}2} ds \lesssim
 \Lambda^{-\frac {\delta}2} e^{\Lambda}. \notag
 \end{align}
  This completes the proof of Lemma \ref{S2lem1}. 
\end{proof}

We shall first prove the local well-posedness of $(GNS)$ in the classical Sobolev spaces. In order to do so,
for $Q(u,v)$  being the bilinear form given by \eqref{S1eq1}, we define $\cB=\cB(u,v)$ via
\begin{equation}\label{S2eq14}
 \left\{\begin{array}{l}
\displaystyle \pa_t \cB 
-\D \cB 
=Q(u,v), \qquad (t,x)\in\R^+\times\R^3, \\
\displaystyle  \cB|_{t=0}=0.
\end{array}\right.
\end{equation}
{In yet other words, $\cB$ is related to $Q(u,v)$ by the integral:
\begin{align} \notag
\cB= \int_0^t e^{(t-s)\Delta} Q(u(s),v(s)) ds.
\end{align}
}

\begin{prop}\label{S2prop2}
{\sl Let $\ga\in [1/2,\infty[,$
$p_\ga\eqdefa\left\{\begin{array}{l}
4, \quad \  \mbox{if}\ \ga>1,\\
\f8{3-2\ga}\  \mbox{if}\ \ga\in [1/2,1],
\end{array}\right.$ and
$u,v$ belong to $L^{p_\ga}([0,T];$ $ H^{\ga+\f2{p_\ga}}).$ Then \eqref{S2eq14} has a unique solution
$\cB(u,v)\in C([0,T]; H^\ga)\cap L^2(]0,T[; \dot H^{\ga+1}).$ Moreover, there holds
\beq \label{S2eq15}
\begin{split}
\|& \cB(u,v)\|_{L^\infty_T(H^\ga)}+\|\na \cB(u,v)\|_{L^2_T(H^\ga)} \\
&\leq C_\ga(T)\|u\|_{L^{p_\ga}_T(H^{\ga+\f2{p_\ga}})}
\|v\|_{L^{p_\ga}_T(H^{\ga+\f2{p_\ga}})} \with C_\ga(T)\eqdefa\left\{\begin{array}{l}
CT^{\f14}, \quad \  \mbox{if}\ \ga>1,\\
T^{\f12-\f2{p_\ga}}\  \mbox{if}\ \ga\in [1/2,1],
\end{array}\right.
\end{split}
\eeq
and
\beq \label{S2eq19}
\|\cB(u,v)\|_{L^{p_\ga}_T(H^{\ga+\f2{p_\ga}})}\leq C_\ga(T)\bigl(1+T^{\f1{p_\ga}}\bigr)\|u\|_{L^{p_\ga}_T(H^{\ga+\f2{p_\ga}})}
\|v\|_{L^{p_\ga}_T(H^{\ga+\f2{p_\ga}})}.
\eeq
}
\end{prop}

\begin{proof} For simplicity, we just present the {\it a priori} estimates
{ \eqref{S2eq15} and \eqref{S2eq19}}.
We first get, by taking $H^\ga$ inner product of \eqref{S2eq14} with $\cB(u,v),$ that
\beq \label{S2eq16}
\f12\f{d}{dt}\|\cB(u,v)(t)\|_{H^\ga}^2+\|\na\cB(u,v)\|_{H^\ga}^2=\bigl(\w{D}^\ga Q(u,v), \w{D}^\ga\cB(u,v)\bigr)_{L^2}.
\eeq
In the case when $\ga>1,$ we deduce from \eqref{S1eq1} and the law of product in Sobolev space (see Theorem 8.3.1 of \cite{H97} for instance), that
\begin{align*}
\bigl|\bigl(\w{D}^\ga Q(u,v), \w{D}^\ga\cB(u,v)\bigr)_{L^2}\bigr|\lesssim &
\||D|^{-\f12}Q(u,v)\|_{H^{\ga}}\||D|^{\f12}\cB(u,v)\|_{H^{\ga}}\\
\lesssim &\|u\|_{H^{\ga+\f12}}\|v\|_{H^{\ga+\f12}}\|\cB(u,v)\|_{H^{\ga}}^{\f12}\|\na\cB(u,v)\|_{H^{\ga}}^{\f12},
\end{align*}
so that
\beq \label{S2eq17}
\begin{split}
\int_0^T&\bigl|\bigl(\w{D}^\ga Q(u,v), \w{D}^\ga\cB(u,v)\bigr)_{L^2}\bigr|\,dt\\
\lesssim & T^{\f14}\|u\|_{L^4_T(H^{\ga+\f12})}\|v\|_{L^4_T(H^{\ga+\f12})}\|\cB(u,v)\|_{L^\infty_T(H^{\ga})}^{\f12}\|\na\cB(u,v)\|_{L^2_T(H^{\ga})}^{\f12}\\
\leq & CT^{\f12}\|u\|_{L^4_T(H^{\ga+\f12})}^2\|v\|_{L^4_T(H^{\ga+\f12})}^2+\f14\bigl(\|\cB(u,v)\|_{L^\infty_T(H^{\ga})}^{2}
+\|\na\cB(u,v)\|_{L^2_T(H^{\ga})}^{2}\bigr).
\end{split}
\eeq

Similarly, in the case when $\ga\in [1/2,1],$ we have
\begin{align*}
\bigl|\bigl(\w{D}^\ga Q(u,v), \w{D}^\ga\cB(u,v)\bigr)_{L^2}\bigr|\lesssim &
\|u\otimes v\|_{H^{\ga}}\|\na \cB(u,v)\|_{H^{\ga}}\\
\lesssim &\|u\|_{H^{\ga+\f2{p_\ga}}}\|v\|_{H^{\ga+\f2{p_\ga}}}\|\na\cB(u,v)\|_{H^{\ga}},
\end{align*}
so that
\beq \label{S2eq18}
\begin{split}
\int_0^T&\bigl|\bigl(\w{D}^\ga Q(u,v), \w{D}^\ga\cB(u,v)\bigr)_{L^2}\bigr|\,dt\\
\lesssim &T^{\f12-\f2{p_\ga}}\|u\|_{L^{p_\ga}_T(H^{\ga+\f2{p_\ga}})}\|v\|_{L^{p_\ga}_T(H^{\ga+\f2{p_\ga}})}\|\na\cB(u,v)\|_{L^2_T(H^{\ga})}\\
\leq & CT^{1-\f4{p_\ga}}\|u\|_{L^{p_\ga}_T(H^{\ga+\f2{p_\ga}})}^2\|v\|_{L^{p_\ga}_T(H^{\ga+\f2{p_\ga}})}^2+\f12\|\na\cB(u,v)\|_{L^2_T(H^{\ga})}^2,
\end{split}\eeq
where we used the fact: $\ga\geq \f12,$ so that $p_\ga\geq 4.$

By integrating \eqref{S2eq16} over $[0,T]$ and then substituting the estimate \eqref{S2eq17} or \eqref{S2eq18} into the resulting inequality,
we obtain \eqref{S2eq15}.

On the other hand, we deduce from the interpolation inequality in {Sobolev} spaces that
\begin{align*}
\|\cB(u,v)\|_{L^{p_\ga}_T(H^{\ga+\f2{p_\ga}})}\leq
&\|\cB(u,v)\|_{L^{\infty}_T(H^{\ga})}^{1-\f2{p_\ga}}\|\w{D}\cB(u,v)\|_{L^{2}_T(H^{\ga})}^{\f2{p_\ga}}\\
\lesssim &\|\cB(u,v)\|_{L^{\infty}_T(H^{\ga})}^{1-\f2{p_\ga}}\bigl(T^{\f12}\|\cB(u,v)\|_{L^{\infty}_T(H^{\ga})}
+\|\na\cB(u,v)\|_{L^{2}_T(H^{\ga})}\bigr)^{\f2{p_\ga}},
\end{align*}
from which and \eqref{S2eq15}, we deduce  \eqref{S2eq19}. This concludes the proof of Proposition \ref{S2prop2}.
\end{proof}

We are now in a position to complete the proof of Theorem \ref{thm1}.

\begin{proof}[Proof of Theorem \ref{thm1}] We divide the proof of Theorem \ref{thm1} into the following three steps:\\

\noindent{\bf Step 1.} The local existence of classical solution.\\

 In view of $(GNS)$ and \eqref{S2eq14}, we can equivalently reformulate $(GNS)$ as \eqref{S2eq20}.
We are going to use the following iteration scheme to construct the approximate solutions of \eqref{S2eq20}.
\beq \label{S2eq21}
u_1\eqdefa e^{t\D}u_{0}\andf u_{n+1}\eqdefa  e^{t\D}u_{0}+\cB(u_n,u_n).
\eeq

Let $p_\ga$ be determined by Proposition \ref{S2prop2}.
Due to $p_\ga>2,$ we get, by using Minkowski inequality, that
\beq\label{S2eq21a}
\begin{split}
\|e^{t\D}u_{0}\|_{L^{p_\ga}_T(H^{\ga+\f2{p_\ga}})}\lesssim &
\bigl\|\w{\xi}^\ga e^{-t|\xi|^2}\widehat{u_0}(\xi)\bigr\|_{L^{p_\ga}_T(L^2_\xi)}+\bigl\|\w{\xi}^\ga |\xi|^{\f2{p_\ga}} e^{-t|\xi|^2}\widehat{u_0}(\xi)\bigr\|_{L^{p_\ga}_T(L^2_\xi)}\\
\lesssim & T^{\f1{p_\ga}}\|u_0\|_{H^\ga}+\bigl\|\w{\xi}^\ga |\xi|^{\f2{p_\ga}} \|e^{-t|\xi|^2}\|_{L^{p_\ga}_T}\widehat{u_0}(\xi)\bigr\|_{L^2_\xi}\\
\leq & C\bigl(1+T^{\f1{p_\ga}}\bigr)\|u_0\|_{H^\ga}.
\end{split}\eeq
Then it follows from \eqref{S2eq19} and the proof of Lemma 5.5 of \cite{bcd} that in order to prove the convergence of the
approximate solution sequence, $\{u_n\},$ constructed by \eqref{S2eq21}, in ${L^{p_\ga}_T(H^{\ga+\f2{p_\ga}})},$
 we need to take the time $T$ to be sufficiently small. Indeed notice from \eqref{S2eq15} that for $\ga\in ]1/2,1],$ $p_\ga>4,$  for $ C_\ga(T)$
 being determined by \eqref{S2eq15}, we can thus
 define
\beq \label{S2eq22}
T_1\eqdefa \sup\bigl\{\ T\leq 1, \quad 4CC_\ga(T)\bigl(1+T^{\f1{p_\ga}}\bigr)^2\|u_0\|_{H^\ga}<1 \ \bigr\}.
\eeq
With {this definition of} $T_1,$ there exists $u\in {L^{p_\ga}_{T_1}(H^{\ga+\f2{p_\ga}})}$ so that $u$ satisfies  \eqref{S2eq20} and
\beq \label{S2eq23}
\lim_{n\to \infty}\|u_n-u\|_{L^{p_\ga}_{T_1}(H^{\ga+\f2{p_\ga}})}=0 \andf \sup_{n\geq 1}\|u_n\|_{L^{p_\ga}_{T_1}(H^{\ga+\f2{p_\ga}})}
\leq 2C\bigl(1+T_1^{\f1{p_\ga}}\bigr)\|u_0\|_{H^\ga},
\eeq
from which, \eqref{S2eq15} and \eqref{S2eq21}, we deduce that
\beq \label{S2eq24}
 \sup_{n\geq 1}\bigl(\|u_n\|_{L^\infty_{T_1}(H^\ga)}+\|\na u_n\|_{L^2_{T_1}(H^\ga)}\bigr)\leq C\|u_0\|_{H^\ga}\bigl(1+\|u_0\|_{H^\ga}\bigr).
 \eeq

 Notice from \eqref{S2eq21} that
 \begin{align*}
 u_{n+1}-u=&\cB(u_n,u_n)-\cB(u,u)\\
 =&B(u_n-u,u_n)+\cB(u,u_n-u),
 \end{align*}
 from which, \eqref{S2eq15} and \eqref{S2eq23}, we deduce that for $T\leq T_1,$
 \beq\label{S2eq32}
 \begin{split}
&\|(u_{n+1}-u)\|_{L^\infty_T(H^\ga)}+\|\na (u_{n+1}-u)\|_{L^2_T(H^\ga)}\\
&\leq C \|u_n-u\|_{L^{p_\ga}_T(H^{\ga+\f2{p_\ga}})}
\Bigl(\|u_n\|_{L^{p_\ga}_T(H^{\ga+\f2{p_\ga}})}+\|u\|_{L^{p_\ga}_T(H^{\ga+\f2{p_\ga}})}\Bigr)\\
&\leq C\|u_0\|_{H^\ga}\|u_n-u\|_{L^{p_\ga}_T(H^{\ga+\f2{p_\ga}})} \to 0\quad \mbox{as}\ \ n\to \infty.
\end{split}
\eeq

\noindent{\bf Step 2.} The uniform estimate of $\|u_n\|_{X_T}.$\\

On the other hand, by virtue of \eqref{S2eq1} and \eqref{S2eq21}, we write
\beq\label{S2eq25}
\begin{split}
\wh{u}_{n+1,\h}(t,\xi)\eqdefa &\wh{u}_{n+1}(t,\xi) 1_{|\xi|\geq 0.01N_1}\\
=&\wh{u}_{n+1}(t,\xi) 1_{0.01N_1\leq |\xi|<0.1N_1}+\wh{e^{t\D}u_0} 1_{|\xi|\geq 0.1N_1}+\wh{\cB(u_n,u_n)}(t,\xi) 1_{|\xi|\geq 0.1N_1},
\end{split}
\eeq
from which, \eqref{S1eq3} and \eqref{S2eq3}, we infer
\begin{align*}
\|u_{n+1}\|_{X_T}\leq & CT^{\frac 12\left(\gamma-\frac 12\right)} \lambda^{-\left(\gamma-\frac 12-\delta\right)}
\| u_{n+1,{\rm h}} \|_{L_T^{\infty}(H^{\gamma})}+\|e^{t\D}u_0\|_{X_T}\\
&+\bigl\| t^{\frac {\delta } 2}1_{|\xi| \ge 0.1 N_1 }
  | \xi |^{\delta+\frac 12} e^{\lambda \frac t {\sqrt T } |\xi| -\frac {\lambda^2 t}{4 T} } \int_0^te^{-(t-s)|\xi|^2}\widehat{Q(u_n,u_n)}(s,\xi)\,ds \bigr\|_{L^\infty_T(L^2_\xi)}.
  \end{align*}
{Note that}
\begin{align*}
\|e^{t\D}u_0\|_{X_T}= &\| t^{\frac {\delta} 2} |\xi|^{\frac 12+\delta} 1_{|\xi| \ge 0.01  N_1 }
e^{-\frac {\lambda^2 t}{4T} +\lambda \frac t {\sqrt T } |\xi| } e^{-t|\xi|^2} \widehat {u_0} (\xi) \|_{L^2_T(L^2_\xi)}  \\
\lesssim &  T^{\frac 12\left(\gamma-\frac 12\right)} \lambda^{-\left(\gamma-\frac 12-\delta\right)}
\| 1_{|\xi |\ge 0.01 N_1 }|\xi|^{\gamma} \widehat {u_0} (\xi) \|_{L^2_T(L^2_\xi)}  \\
\lesssim & T^{\frac 12\left(\gamma-\frac 12\right)} \lambda^{-(\gamma-\frac 12-\delta)}
\| u_{\rm h} \|_{L_T^{\infty}(H^{\gamma})},
\end{align*}
{where in the last step we invoke the limit solution $u$ determined by \eqref{S2eq23} (note that by
\eqref{S2eq24} $u\in L_T^{\infty} H^{\gamma}$).
Using Proposition \ref{S2prop1} for the nonlinear piece, we obtain}
\begin{align*}
\|u_{n+1}\|_{X_T} \leq &   C\Bigl(T^{\frac 12(\gamma-\frac 12)} \lambda^{-(\gamma-\frac 12-\delta)}\bigl[
\bigl( \| u_{n+1,{\rm h}} \|_{L_T^{\infty}(H^{\gamma})}+   \| u_{\rm h} \|_{L_T^{\infty}(H^{\gamma})}\bigr)
\\
&+ \bigl(1+e^{4\eta_0\la^2}\bigr)T^{\d}
\|u_{n}\|_{L_T^{\infty}(H^{\gamma}) }
\|u_{n,{\rm h}}\|_{L_T^{\infty}(H^{\gamma}) } \bigr]+\lambda^{-\delta} e^{\frac {\lambda^2}{4}} \|u_n\|_{X_T}^2 \\
&
+ \lambda^{2-\delta} T^{\frac {\delta} 2}\bigl(1+\la^{\f12+\d+\ga}T^{\f12\left(\ga-\f12-\d\right)}e^{0.01\la^2}\bigr)  \|u_{n}\|_{L_T^{\infty}(H^{\gamma}) } \|u_n\|_{X_T}\Bigr).
\end{align*}
Then we deduce from \eqref{S2eq32} that for arbitrary but fixed small enough constant $\ve>0,$ there exists an integer ${\rm N}$ so that for any $n\geq{\rm N},$ there holds
\beq \label{S2eq25a}
\begin{split}
\|u_{n+1}\|_{X_T} \leq &   C\Bigl(T^{\frac 12(\gamma-\frac 12)} \lambda^{-(\gamma-\frac 12-\delta)}\bigl[
1+ \bigl(1+e^{4\eta_0\la^2}\bigr)T^{\d}
\|u_{0}\|_{H^{\gamma} }\bigr]\\
&\times\bigl(\ve+
\|u_{{\rm h}}\|_{L_T^{\infty}(H^{\gamma}) } \bigr)+\lambda^{-\delta} e^{\frac {\lambda^2}{4}} \|u_n\|_{X_T}^2 \\
&
+ \lambda^{2-\delta} T^{\frac {\delta} 2}\bigl(1+\la^{\f12+\d+\ga}T^{\f12\left(\ga-\f12-\d\right)}e^{0.01\la^2}\bigr)  \|u_{0}\|_{H^{\gamma}} \|u_n\|_{X_T}\Bigr).
\end{split}\eeq

Before proceeding, we present the following lemma:

\begin{lem}[{Uniform high frequency smallness in $H^{\gamma}$}]\label{S2lem3}
{\sl Let $\gamma>\frac 12$ and $u \in C([0,
{T}]; H^{\gamma})$ be a mild solution of $(GNS)$ with
initial data $u_0 \in H^{\gamma}.$ We denote the Fourier multiplier $\widehat{P_{> J}}(\xi) \eqdefa 1_{|\xi| \ge 0.01  J}.$ Then one has
\beq \label{1.34ca}
\eta_J^\ga(T) \eqdefa \max_{0\le t \le {T} } \| P_{> J} u (t)\|_{H^{\gamma}} \to 0, \qquad
\text{as}\ \  J\to \infty.
\eeq}
\end{lem}

We admit the lemma for the time being and postpone its proof till the end of this section.

We now continue our proof of Theorem \ref{thm1}. For $\eta_0<\f{\d}{8(2\ga-1)}$ and ${\eta_J^\ga(t)}$ being given by \eqref{S1eq2}, we define
\beq\label{S2eq26}
\begin{split}
\lambda_\ve(T)\eqdefa&\sqrt{ (2\gamma-1) \bigl(|\ln T|  + \ln |\ln T|\bigr) +3 \beta_\ve(T) }\with\\
\beta_\ve(T)\eqdefa&
\min\bigl\{\  |\ln \bigl(\ve+\eta_{T^{-\frac 12}}^\ga(T)\bigr)  |, \, \frac 12\bigl(\gamma-\frac 12\bigr) |\ln T|\ \bigr\},
\end{split}
\eeq
and
\beq \label{S2eq27}
\begin{split}
T_2\eqdefa \sup\Bigl\{\ T&\leq T_1,\   e^{4\eta_0\la_\ve^2(T)}T^\d\leq 1,\\
 &C\lambda^{2-\delta}_\ve(T) T^{\frac {\delta} 2}\bigl(1+\la^{\f12+\d+\ga}_\ve(T)T^{\f12\left(\ga-\f12-\d\right)}e^{0.01\la^2_\ve(T)}\bigr) \|u_0\|_{H^\ga}\leq \f12 \ \Bigr\}.
 \end{split}
\eeq
We observe that due to $\ga>\f12+\d,$ $T_2$ is well-defined for $\d$ being sufficiently  small.

Then by virtue of \eqref{S2eq25a}, for any $n\geq{\rm N}$  and $T\leq T_2,$ we achieve
\beq \label{S2eq28}
\begin{split}
\|u_{n+1}\|_{X_T} \leq &   C\Bigl(T^{\frac 12(\gamma-\frac 12)} \lambda_\ve^{-(\gamma-\frac 12-\delta)}(T)\bigl(1+\|u_{0}\|_{H^{\gamma} }\bigr)
\bigl(\ve+   \| u_{\rm h} \|_{L_t^{\infty}(H^{\gamma})}\bigr)
\\
& +\lambda_\ve^{-\delta}(T) e^{\frac {\lambda_\ve^2(T)}{4}} \|u_n\|_{X_T}^2\Bigr)
+ \f12 \|u_n\|_{X_T}.
\end{split}
\eeq
Let us compare $\|u_{n+1}\|_{X_T} $ with $Z,$ which is determined by
\beno
2Z= C\Bigl(T^{\frac 12\left(\gamma-\frac 12\right)} \lambda_\ve^{-\left(\gamma-\frac 12-\delta\right)}(T)\bigl(1+\|u_{0}\|_{H^{\gamma} }\bigr)
\bigl(\ve+   \| u_{\rm h} \|_{L_T^{\infty}(H^{\gamma})}\bigr)
+\lambda_\ve^{-\delta}(T) e^{\frac {\lambda^2_\ve(T)}{4}}Z^2\Bigr).
\eeno
In particular, thanks to Lemma \ref{S2lem3} and \eqref{S2eq26}, for any sufficiently small  positive constant $c_0,$ which will be determined later on, we can shrink
$t_0\leq T_2$ to be so small  that
\beq\label{S2eq29}
t_0\eqdefa\sup \bigl\{ T\leq T_2,
C^2T^{\frac 12\left(\gamma-\frac 12\right)} \lambda_\ve^{-\left(\gamma-\frac 12\right)}(T)\bigl(1+\|u_{0}\|_{H^{\gamma} }\bigr)
\bigl(\ve+   \| u_{\rm h} \|_{L_T^{\infty}(H^{\gamma})}\bigr) e^{\frac {\lambda_\ve^2(T)}{4}}\leq c_0 \bigr\}.
\eeq
We notice that for $\la_\ve(T)$ defined by \eqref{S2eq26}, $t_0$ defined by above can be reached.

Then we deduce that for $n\geq{\rm N}$ and $T\leq t_0,$
\beq\label{S2eq31a}
\|u_{n+1}\|_{X_T} \leq CT^{\frac 12\left(\gamma-\frac 12\right)} \lambda_\ve^{-\left(\gamma-\frac 12-\delta\right)}(T)\bigl(1+\|u_{0}\|_{H^{\gamma} }\bigr)
\bigl(\ve+   \| u_{\rm h} \|_{L_T^{\infty}(H^{\gamma})}\bigr).
\eeq

\noindent{\bf Step 3.} The convergence of  $\{t^{\f\d2}\T(u_n)\}$ in $L^\infty([0,t_0]; \dot H^{\f12+\d}).$\\

Indeed we deduce
 from \eqref{S2eq1} and Lemma \ref{S2lem2} that for $T\leq t_0,$
\beq\label{S2eq30}
\begin{split}
\bigl\|t^{\f{\d}2}\T(u_{n+1}-u_n)\bigr\|_{L^\infty_T(\dot H^{\f12+\d})}
\leq &\bigl\|t^{\f{\d}2}\T\left((u_{n+1}-u_n)_{\rm l}\right)\bigr\|_{L^\infty_T(\dot H^{\f12+\d})}+\|u_{n+1}-u_n\|_{X_T}\\
\leq &C\|u_{n+1}-u_n\|_{L^\infty_T(H^\ga)}+\|u_{n+1}-u_n\|_{X_T}.
\end{split}
\eeq
Notice from \eqref{S2eq25} that
\begin{align*}
\|u_{n+1}-u_n\|_{X_T}\leq & \bigl\| t^{\frac {\delta } 2} | \xi |^{\delta+\frac 12} 1_{0.01N_1\leq |\xi| < 0.1 N_1 }
   e^{\lambda \frac t {\sqrt T } |\xi| -\frac {\lambda^2 t}{4 T} } (\wh{u}_{n+1}-\wh{u}_n)\bigr\|_{L^\infty_T(L^2_\xi)}\\
  &+\bigl\| t^{\frac {\delta } 2}1_{|\xi| \ge 0.1 N_1 }
  | \xi |^{\delta+\frac 12} e^{\lambda \frac t {\sqrt T } |\xi| -\frac {\lambda^2 t}{4 T} } \int_0^te^{-(t-s)|\xi|^2}
  \bigl(\widehat{Q}(u_{n}-u_{n-1},u_{n})\\
  &\qquad\qquad\qquad\qquad\qquad\qquad\qquad+\widehat{Q}(u_{n-1},u_{n}-u_{n-1})\bigr)(s,\xi)\,ds \bigr\|_{L^\infty_T(L^2_\xi)}
  \end{align*}
  from which,  \eqref{S2eq3} and Proposition \ref{S2prop2}, we infer
  \begin{align*}
\|u_{n+1}-u_n\|_{X_T}\leq & C\Bigl(\bigl(T^{\frac 12\left(\gamma-\frac 12\right)} \lambda_\ve^{-\left(\gamma-\frac 12-\delta\right)}(T)
\bigl[\|u_{n+1}-u_{n}\|_{L^\infty_T(H^\ga)}\\
&+T^\d
\bigl(\|u_n\|_{L^\infty_T(H^\ga)}+\|u_{n-1}\|_{L^\infty_T(H^\ga)}\bigr)\|u_{n}-u_{n-1}\|_{L^\infty_T(H^\ga)}\bigr]\\
&+\la_\ve^{2-\d}(T)T^{\f\d2}\bigl(\|u_{n}\|_{X_T}+\|u_{n-1}\|_{X_T}\bigr)
\|u_{n}-u_{n-1}\|_{L^\infty_T(H^\ga)}\\
&+\|u_{n}-u_{n-1}\|_{X_T}
\bigl[\la_\ve^{-\d}(T)e^{\f{\la_\ve^2(T)}4}(\|u_n\|_{X_T}+\|u_{n-1}\|_{X_T})\\
&\qquad\qquad\qquad\qquad+\la_\ve^{2-\d}T^{\f\d2}(\|u_{n}\|_{L^\infty_T(H^\ga)}+
\|u_{n-1}\|_{L^\infty_T(H^\ga)})\bigr] \Bigr).
\end{align*}
It follows from \eqref{S2eq29} and \eqref{S2eq31a} that for $T\leq t_0$
\begin{align*}
\la_\ve^{-\d}e^{\f{\la_\ve^2(T)}4}(\|u_n\|_{X_T}+\|u_{n-1}\|_{X_T})
\leq &2CT^{\frac 12\left(\gamma-\frac 12\right)} \lambda_\ve^{-\left(\gamma-\frac 12\right)}(T)\\
&\times\bigl(1+\|u_{0}\|_{H^{\gamma} }\bigr)
\bigl(\ve+   \| u_{\rm h} \|_{L_T^{\infty}(H^{\gamma})}\bigr)e^{\f{\la_\ve^2(T)}4}\leq \f2{C}c_0.
\end{align*}
Then if we
 take $c_0$ to be so small that $c_0\leq \f18,$
we deduce from \eqref{S2eq24} that for $T\leq t_0,$
\begin{align*}
\|u_{n+1}-u_n\|_{X_T}\leq & C\bigl(\|u_{n+1}-u_{n}\|_{L^\infty_T(H^\ga)}+\|u_{0}\|_{H^{\gamma} }\|u_{n}-u_{n-1}\|_{L^\infty_T(H^\ga)}\bigr)\\
&+\bigl(\f14+C\la_\ve^{2-\d}(T)T^{\f\d2}\|u_{0}\|_{H^{\gamma} }\bigr)\|u_{n}-u_{n-1}\|_{X_T},
\end{align*}
from which and
 \eqref{S2eq27}, we infer for $ T\leq t_0$
\beno
\|u_{n+1}-u_n\|_{X_T}\leq C\bigl(\|u_{n+1}-u_n\|_{L^\infty_T(H^\ga)}+\|u_{n}-u_{n-1}\|_{L^\infty_T(H^\ga)}\bigr)+\f12\|u_{n}-u_{n-1}\|_{X_T}.
\eeno
By inserting the above estimate into \eqref{S2eq30}, we find for $T\leq t_0,$
\beq \label{S2eq31}
\begin{split}
\bigl\|t^{\f{\d}2}\T(u_{n+1}-u_n)\bigr\|_{L^\infty_T(\dot H^{\f12+\d})}
\leq &\f12\bigl\|t^{\f{\d}2}\T(u_{n}-u_{n-1})\bigr\|_{L^\infty_T(\dot H^{\f12+\d})} \\
&+C\bigl(\|u_{n+1}-u_n\|_{L^\infty_T(H^\ga)}+\|u_{n}-u_{n-1}\|_{L^\infty_T(H^\ga)}\bigr).
\end{split}
\eeq
Hence $\bigl\{ t^{\f\d2}\T(u_n) \bigr\}$ is a Cauchy sequence in $L^\infty_T(\dot H^{\f12+\d})$ for any $T\leq t_0.$ As a result, it { follows that}
\beq\label{S2eq34}
\lim_{n\to\infty}\bigl\|t^{\f{\d}2}\T(u_{n+1}-u)\bigr\|_{L^\infty_T(\dot H^{\f12+\d})}=0 \andf
\bigl\|t^{\f\d2}e^{-\f{\la^2_\ve t}{4T}}e^{\f{\la_\ve t}{\sqrt{T}}|D|}u(t)\bigr\|_{L^\infty_T(\dot H^{\f12+\d})}\leq C,
\eeq for any $T\leq t_0.$

In particular, by taking $t=T$ in \eqref{S2eq34}, we obtain
\beno
\bigl\|e^{{\la_\ve(T) \sqrt{T}}|D|}u(T)\bigr\|_{\dot H^{\f12+\d}}\leq C T^{-\f\d2}e^{\f{\la_\ve^2(T)}4}\quad \mbox{for any} \ T\leq t_0.
\eeno
By taking $\e\to 0$ in the above inequality, we obtain
\beno 
\bigl\|e^{{\la(T) \sqrt{T}}|D|}u(T)\bigr\|_{\dot H^{\f12+\d}}\leq C T^{-\f\d2}e^{\f{\la^2(T)}4},
\eeno
{for any}  $ T\leq t_0$ and $\la(T)$ being given by \eqref{S1eq2}, which together with \eqref{S1eq2} ensures \eqref{S2eq35}. This completes the proof of Theorem \ref{thm1}.
\end{proof}

Let us now present the proof of Lemma \ref{S2lem3}.

\begin{proof}[Proof of Lemma \ref{S2lem3}]
We first observe that
 the linear part satisfies
\begin{align*}
\sup_{t\geq 0}\bigl\| P_{> J} e^{t\Delta} u_0 \bigr\|_{H^{\gamma}} \le \| P_{>J} u_0 \|_{H^{\gamma}} \to 0,
\qquad \text{as $J\to \infty$}.
\end{align*}

For the nonlinear part, we first deal with the case when $\gamma \in ]1/2, 3/2[.$ Then for $J\geq 1$ and sufficiently small constant, $\epsilon>0,$
satisfying $\ga>\f12+\epsilon,$
we get, by using the law of product in Sobolev spaces (see Theorem 8.3.1 of \cite{H97} for instance {or see
\cite{LiRMI19}}), that
\begin{align*}
& \bigl\| \int_0^t P_{>J}  \w{D}^{\gamma}  e^{(t-s) \Delta } Q(u,u)(s) \,ds \bigr\|_{L^2} \\
&\lesssim  J^{-\epsilon}  \int_0^t \bigl\| |D|^{1+\epsilon+\gamma} P_{>J}e^{(t-s)\Delta} ( u\otimes u)(s) \bigr\|_{L^2}\, ds  \\
&\lesssim  J^{-\epsilon}  \int_0^t \bigl\| |D|^{\f52-\gamma+\epsilon} P_{>J}e^{(t-s)\Delta}|D|^{2\ga-\f32} ( u\otimes u)(s) \bigr\|_{L^2}\, ds  \\
& \lesssim  J^{-\epsilon} \int_0^t  (t-s)^{-\frac12 \left(\frac 52 -\gamma+\epsilon\right)}\| u(s) \|_{H^{\gamma} }^2 \,ds\\
  & \lesssim  J^{-\epsilon} t^{\f12\left(\ga-\f12-\epsilon\right)} \| u\|_{L^\infty_t(H^{\gamma})}^2
  \to 0,
 \qquad \text{as $J\to \infty$},
 \end{align*} where in the last step, we used the fact that $\ga>\f12+\epsilon.$

For $\gamma=\frac 32$,  we take $\epsilon\in ]0,1[.$
{By a slight modification of the above argument, we have}
\begin{align*}
& \bigl\| \int_0^t P_{>J}  \w{D}^{\gamma}  e^{(t-s) \Delta } Q(u,u)(s) \,ds \bigr\|_{L^2} \\
 &\lesssim J^{-\epsilon}  \int_0^t \bigl\| |D|^{1+\epsilon+\frac 32} P_{>J}e^{(t-s)\Delta} ( u\otimes u)(s) \bigr\|_{L^2}\, ds\\
 &\lesssim J^{-\epsilon}  \int_0^t \bigl\| |D|^{1+2\epsilon} P_{>J}e^{(t-s)\Delta}\w{D}^{\f32-\e} ( u\otimes u)(s) \bigr\|_{L^2} \,ds\\
& \lesssim  J^{-\epsilon} \int_0^t  (t-s)^{-\frac12 (1+2\epsilon)}\| u(s) \|_{H^{\frac 32} }^2 \, ds  \lesssim J^{-\epsilon}t^{\f12-\epsilon}
\| u\|_{L^\infty_t(H^{\frac 32})}^2 \to 0,
 \qquad \text{as $J\to \infty$}.
 \end{align*}

Finally for $\gamma>\frac 32$, we note that $H^{\gamma}$ is  an algebra. As a consequence, we deduce that
\begin{align*}
& \bigl\| \int_0^t P_{>J}  \w{D}^{\gamma}  e^{(t-s) \Delta } Q(u,u)(s) \,ds \bigr\|_{L^2} \\
& \lesssim J^{-\epsilon}  \int_0^t \bigl\| |D|^{1+\epsilon+\gamma} P_{>J}e^{(t-s)\Delta} ( u\otimes u)(s)\bigr\|_2 \,ds  \\
& \lesssim J^{-\epsilon} \int_0^t  (t-s)^{-\frac12 (1+\epsilon)}\| u(s) \|_{H^{\gamma} }^2 \, ds  \lesssim J^{-\epsilon}t^{\f12(1-\epsilon)} \| u\|_{L^\infty_t(H^{\gamma})}^2 \to 0,
 \qquad \text{as $J\to \infty$}.
 \end{align*}

{Collecting the above estimates}, we conclude the proof of \eqref{1.34ca}.
\end{proof}

\vskip 0.2cm

\setcounter{equation}{0}

\section{The critical case: $\gamma=\frac 12$}
 The goal of
this section is to present the proof of Theorem \ref{thm2}, namely,  the critical case of Theorem \ref{thm1} for $\gamma=\frac 12$.
In order to do so, we first { need to adapt}  Lemma \ref{S2lem3} to the critical case  $\ga=\f12.$

\begin{prop}[Uniform high frequency smallness {in} $\dot H^{\frac 12}$]\label{S3prop1}
{\sl Let $u_0 \in \dot H^{\frac 12}$.  There exists a positive time
 ${T}={T}(u_0)$ so that the system $(GNS)$ with initial data $u_0$ has  a unique solution $u\in
C([0, {T}]; \dot H^{\frac 12})\cap L^2(]0,T[; \dot H^{\f32})$ and
\beq\label{S3eq1}
\begin{split}
\|u&\|_{L_{T}^{\infty}  (\dot{H}^{\f12})}+\|\na u\|_{L^2_{T}(\dot{H}^{\f12})}+\| t^{\frac {m+\delta}2}u \|_{L_{T}^{\infty}(\dot{H}^{m+\d+\f12})}\\
&+\| t^{\frac {m+\delta}2}\na u \|_{L_{T}^{2}(\dot{H}^{m+\d+\f12})}
 \le {C} \| u_0\|_{\dot{H}^{\frac12} }\quad \mbox{for any integer} \ m\geq 0\andf \d\in ]0,1[.
\end{split}\eeq
 Furthermore there holds
 \begin{align}  \label{eta2.2a}
\zeta_J^{\f12}(T)\eqdefa \max_{0\le t \le {T} } \| 1_{|\xi| \ge J} |\xi|^{\frac 12}
\widehat{ u }(t,\xi)\|_{L_{\xi}^2} \to 0, \quad
\text{as} \ J\to \infty.
\end{align}}
\end{prop}
\begin{rmk}
It is possible to obtain all polynomial smoothing estimates in \eqref{S3eq1} at one stroke if we take $T=T(u_0)$
to be { sufficiently small}.
We sketch an alternative argument as follows. The key is to estimate the following $Z$-norm:
\begin{align}
\| u\|_{Z_T}\eqdefa\| t^{\frac 18} |D|^{\frac 34} v \|_{L_T^{\infty}(L^2)}, \with
\widehat v (t,\xi)\eqdefa e^{ \frac {\sqrt t}2 |\xi| } \widehat u(t,\xi).
\end{align}
For the linear part, it is not difficult to check that
\beq\label{3.5ec1}
\begin{split}
& t^{\frac 18} \| |D|^{\frac 34} e^{ \frac{\sqrt t }2|D|} e^{t \Delta} u_0 \|_{L^\infty(\R^+;L^2)}
 \lesssim \| u_0\|_{\dot H^{\frac 12}};\\
&
\lim_{t\to 0^+} t^{\frac 18} \| |D|^{\frac 34} e^{ \frac{\sqrt t}2 |D|} e^{t \Delta} u_0 \|_{L^2} =0.
\end{split}\eeq

On the other hand, for the nonlinear part, we first get, by using the {
triangle} inequality, $|\xi|\leq |\xi-\eta|+|\eta|$ for $\xi,\eta\in\R^3,$ that
\begin{align*}
 &  t^{\frac 18} \int_0^t \bigl\|e^{ \frac{\sqrt t}2 |D|} |D|^{\frac 34} e^{(t-s) \Delta} Q(u(s), u(s) ) \bigr\|_{L^2} \,ds  \\
 & \leq t^{\frac 18} \int_0^t \bigl\||\xi|^{\f74}e^{ \frac12({\sqrt t}-{\sqrt s})|\xi|}  e^{-(t-s)|\xi|^2}|\widehat{v}(s)|\ast|\widehat v(s)| \bigr\|_{L^2_\xi} \,ds  \\
  & \lesssim t^{\frac 18} \int_0^t \bigl\||\xi|^{\f74}  e^{-\f12(t-s)|\xi|^2}|\widehat{v}(s)|\ast|\widehat v(s)|  \bigr\|_{L^2_\xi} \,ds  \\
& \lesssim  t^{\frac 18} \int_0^t (t-s)^{-\frac 78} \| |\widehat{v}(s)|\ast|\widehat v(s)| \|_{L_{\xi}^2}\, ds\\
 &\lesssim   t^{\frac 18} \int_0^t (t-s)^{-\frac 78} \bigl\|\mathcal{F}_{\xi\to x}(|\wh{v}(s,\xi)|)\big\|_{L^{4}}^2\,ds.
  \end{align*}
 Then by applying Sobolev imbedding inequality, $H^{\f34}(\R^3)\hookrightarrow L^4(\R^3),$ we obtain
  \begin{align*}
   t^{\frac 18} \int_0^t \bigl\|e^{ \frac{\sqrt t}2 |D|} |D|^{\frac 34} e^{(t-s) \Delta} Q(u(s), u(s) ) \bigr\|_{L^2} \,ds
 &\lesssim  t^{\frac 18} \int_0^t (t-s)^{-\frac 78} \bigl\| |D|^{\frac 34}\mathcal{F}_{\xi\to x}(|\wh{v}(s,\xi)|)\bigr\|_{L^2}^2 \,ds
  \\
& \lesssim t^{\frac 18} \int_0^t (t-s)^{-\frac 78}s^{-\f14}\,ds \| u\|_{Z_T}^2\\
&\lesssim \| u\|_{Z_T}^2.
 \end{align*}
from which and  \eqref{3.5ec1}, we deduce from  the fixed point theorem (see Lemma 5.5 of \cite{bcd} for instance) that
there exists $T$ so that \eqref{S2eq20} has a unique solution $u$ on $[0,T]$ with
 $\|u\|_{Z_T} \lesssim \|u_0 \|_{\dot H^{\frac 12}},$ which in particular
implies all polynomial smoothing estimates.
\end{rmk}

\begin{proof}[Proof of Proposition \ref{S3prop1}]
 We first get, by a similar derivation of \eqref{S2eq21a}, that
\begin{align*}
\|e^{t\D}u_0\|_{L^4(\R^+;\dot{H}^1)}=&\bigl\||\xi|e^{-t|\xi|^2}\wh{u}_0(\xi)\bigr\|_{L^4(\R^+;L^2_\xi)}\\
\leq &\bigl\||\xi|\|e^{-t|\xi|^2}\|_{L_t^4(\R^+)}\wh{u}_0(\xi)\bigr\|_{L^2_\xi}
\leq C\|u_0\|_{\dot{H}^{\f12}},
\end{align*}
which in particular implies
\beq \label{S3eq2}
\lim_{T\to 0}\|e^{t\D}u_0\|_{L^4_T(\dot{H}^1)}=0.
\eeq
{For $\cB(u,v)$ being determined by \eqref{S2eq14}}, we deduce from  \eqref{S2eq19} that
\beno
\|\cB(u,v)\|_{L^4_T(\dot{H}^1)}\leq C_0\|u\|_{L^4_T(\dot{H}^1)}\|v\|_{L^4_T(\dot{H}^1)}\quad\mbox{if}\ \ T\leq 1.
\eeno
{ Take $T(u_0)>0$ sufficiently small such that}
\beno
4C_0\|e^{t\D}u_0\|_{L^4_T(\dot{H}^1)}<1.
\eeno
{ The} fixed point theorem (see Lemma 5.5 of \cite{bcd} for instance) ensures that
\eqref{S2eq20} {admits} a unique solution $u$ in $L^4(]0,T[; \dot{H}^1).$ Moreover, it follows from \eqref{S2eq15} that
\beq\label{S3eq3}
\begin{split}
&\|u\|_{L^\infty_T(\dot{H}^{\f12})}+\|\na u\|_{L^2_T(\dot{H}^{\f12})}\leq C(\|u_0\|_{\dot{H}^{\f12}}) \andf \\
&\|u\|_{L^q_T(\dot{H}^{\f12+\f2q})}\leq \|u\|_{L^\infty_T(\dot{H}^{\f12})}^{1-\f2q}\|u\|_{L^2_T(\dot{H}^{\f32})}^{\f2q}\leq C(\|u_0\|_{\dot{H}^{\f12}})
\quad \mbox{for any}\ \ q\in [2,\infty].
\end{split}
\eeq

For any $p\in ]1,\infty[,$ we choose $q\in ]2,\infty[$ with $p>\f{q}2.$ Then we get, by applying the law of product in Sobolev space, that
\begin{align*}
\|\cB(u,u)(t)\|_{\dot{H}^{\f12+\f2p}}\lesssim
&\int_0^t\bigl\||{\xi}|^{\f12+\f2p}e^{-(t-s)|\xi|^2}|\xi|\wh{u\otimes u}(s)\bigr\|_{L^2}\,ds\\
\lesssim
&\int_0^t\bigl\||{\xi}^{2+\f2p-\f4q}e^{-(t-s)|\xi|^2}|{\xi}|^{\f4q-\f12}\wh{u\otimes u}(s)\bigr\|_{L^2}\,ds\\
\lesssim
&\int_0^t(t-s)^{-1-\f1p+\f2q}\|u(s)\|_{\dot{H}^{\f12+\f2q}}^2\,ds,
\end{align*}
from which, \eqref{S3eq3} and Hardy-Littlewood-Sobolev inequality, we infer
\begin{align*}
\|\cB(u,u)\|_{L^p_T(\dot{H}^{\f12+\f2p})}\lesssim &\bigl\|\int_0^\infty |t-s|^{-1-\f1p+\f2q}\chi_{[0,T]}(s)\|u(s)\|_{\dot{H}^{\f12+\f2q}}^2\,ds
\bigr\|_{L^p}\\
\lesssim &\|u\|_{L^q_T(\dot{H}^{\f12+\f2q})}^2\leq C(\|u_0\|_{\dot{H}^{\f12}}).
\end{align*}
Hence, by applying interpolation inequality, we obtain
\beq\label{S3eq3a}
\|\cB(u,u)\|_{L^2_T(\dot B^{\f32}_{2,1})}\lesssim \|\cB(u,u)\|_{L^4_T(\dot{H}^1)}^{\f12}\|\cB(u,u)\|_{L^{\f43}_T(\dot H^2)}^{\f12}\leq C(\|u_0\|_{H^{\f12}}),
\eeq
where $\|a\|_{\dot B^{\f32}_{2,1}}$ denotes
 the homogeneous Besov norm of $a$ in the space $\dot B^{\f32}_{2,1}$ (see Definition 2.15 of \cite{bcd};
 {see also Lemma 2.7 of \cite{LiRMI19} and the discussion therein for
 more general interpolation inequalities}).

Whereas it follows from Theorem 2.34 that
\begin{align*}
\|e^{t\D}u_0\|_{L^2(\R^+;L^\infty)}=\Bigl(\int_0^\infty \bigl\|t^{\f12}e^{t\D}u_0\|_{L^\infty}^2\f{dt}{t}\Bigr)^{\f12}
=\|u_0\|_{\dot B^{-1}_{\infty,2}}\leq C\|u_0\|_{\dot{H}^{\f12}},
\end{align*}
which together with  \eqref{S3eq3a} ensures that
\beq \label{S3eq4}
\|u\|_{L^2_T(L^\infty)}\leq C(\|u_0\|_{\dot{H}^{\f12}}).
\eeq

Next let us turn to the time-weighted energy estimate. For simplicity, we just present the
{\it a  priori } estimate. Indeed, for any $\d\in ]0,1[,$ we get, by taking $\dot{H}^{m+\d+\f12}$ inner product of $(GNS)$ with
$u,$ that
\beq \label{S3eq4a}
\begin{split}
\f12\f{d}{dt}\|u(t)\|_{\dot{H}^{m+\d+\f12}}^2+\|\na u\|_{\dot{H}^{m+\d+\f12}}^2=
&\bigl(|{D}|^{m+\d+\f12}Q(u,u), |{D}|^{m+\d+\f12}u\bigr)_{L^2}\\
\lesssim &\|u\otimes u\|_{\dot{H}^{m+\d+\f12}}\|\na u\|_{\dot{H}^{m+\d+\f12}}\\
\leq &C\|u\|_{L^\infty}^2\|u\|_{\dot{H}^{m+\d+\f12}}^2+\f12\|\na u\|_{\dot{H}^{m+\d+\f12}}^2,
\end{split} \eeq
where we used Bony's decomposition \cite{Bo} that $u\otimes u=T_u\otimes u+u\otimes T_u+R(u,u)$ so that
\beno \|u\otimes u\|_{\dot{H}^{m+\d+\f12}}
\leq C\|u\|_{L^\infty}\|u\|_{\dot{H}^{m+\d+\f12}}. \eeno

{Multiplying} the inequality \eqref{S3eq4a} by $t^{m+\d},$ we obtain
\begin{align*}
&\f{d}{dt}\|t^{\f{m+\d}2}u(t)\|_{\dot{H}^{m+\d+\f12}}^2+\|t^{\f{m+\d}2}\na u\|_{\dot{H}^{m+\d+\f12}}^2\\
&\leq (m+\d)
\|t^{\f{m-1+\d}2}u(t)\|_{\dot{H}^{m+\d+\f12}}^2+C\|u\|_{L^\infty}^2\|t^{\f{m+\d}2}u(t)\|_{\dot{H}^{m+\d+\f12}}^2.
\end{align*}
Applying Gronwall's inequality gives rise to
\beq \label{S3eq5}
\begin{split}
\|t^{\f{m+\d}2}u\|_{L^\infty_t(\dot{H}^{m+\d+\f12})}^2&+\|t^{\f{m+\d}2}\na u\|_{L^2_t(\dot{H}^{m+\d+\f12})}^2\leq \Bigl(\|t^{\f{m-1+\d}2}\na u(t)\|_{L^2_t(\dot{H}^{m+\d-\f12})}^2\\
&\qquad+(m+\d)\|t^{\f{m-1+\d}2} u(t)\|_{L^2_t(\dot{H}^{m+\d-\f12})}^2\Bigr)
\exp\bigl(C\|u\|_{L^2_t(L^\infty)}^2\bigr).
\end{split}
\eeq

With \eqref{S3eq4} and \eqref{S3eq5}, to conclude the proof of the time-weighted estimate part in \eqref{S3eq1} via
induction argument, we still need the following lemma, the proof of which will be postponed after we finish the proof
of Proposition \ref{S3prop1}.

\begin{lem}\label{S3lem1}
{\sl Under the assumptions of Proposition \ref{S3prop1}, for any $t\leq T(u_0),$ one has
\beq \label{S3eq6a}
\|t^{\f{\d}2}u\|_{L^\infty_t(\dot{H}^{\d+\f12})}^2+\|t^{\f{\d}2}\na u\|_{L^2_t(\dot{H}^{\d+\f12})}^2\leq C(\|u_0\|_{\dot{H}^{\f12}}).
\eeq}
\end{lem}

Finally,
we present the proof of  \eqref{eta2.2a}.  Due to $u \in C([0,T]; \dot{H}^{\frac 12})$, for any $\epsilon>0$,  we can
choose $\tau_0>0$ so small that
\begin{align*}
 \max_{0\le t\le \tau_0} \| u(t) - u_0 \|_{\dot{H}^{\frac 12} } \le 0.001 \epsilon,
 \end{align*}
which implies that for any $J\ge 1$,
\begin{align*}
\max_{0\le t \le \tau_0} \| 1_{|\xi| \ge J} |\xi|^{\frac 12} ( \widehat{u}(t, \xi) - \widehat{ u_0 }(\xi) ) \|_{L^2_\xi} \le 0.01
\epsilon.
\end{align*}
While by taking $J_0\ge 1$ so large that
\begin{align*}
\| 1_{|\xi| \ge J_0} |\xi|^{\frac 12}  \widehat{u_0}(\xi)  \|_{L^2_\xi} \le 0.01 \epsilon,
\end{align*}
we deduce that for any $J\ge J_0$,
\beq\label{S3eq8}
\max_{0\le t\le \tau_0 } \| 1_{|\xi| \ge J} |\xi|^{\frac 12} \widehat{u}(t,\xi) \|_{L^2_\xi} \le 0.02 \epsilon.
\eeq

On the other hand, it follows from \eqref{S3eq1} that $t^{\f34}u\in L^\infty([0,T]; \dot{H}^2),$ from which,
we infer
\begin{align*}
\max_{\tau_0\le t \le T} \| P_{>J} u(t, \cdot) \|_{\dot{H}^{\frac 12}} \le C \tau_0^{-\f34} J^{-\frac 32}
\to 0 \qquad \text{as $J\to \infty$}.
\end{align*}
In particular, we can take   $J_1\ge 1$ so large that
\begin{align*}
\max_{\tau_0 \le t \le {T_0}} \| 1_{|\xi| \ge J_1} |\xi|^{\frac 12} \widehat{u}(t,\xi) \|_{L^2_\xi} \le 0.1 \epsilon,
\end{align*}
which together with \eqref{S3eq8} ensures that
 for $J\ge \max\{J_0, \, J_1\}$,
\begin{align*}
\max_{0 \le t \le {T}} \| 1_{|\xi| \ge J} |\xi|^{\frac 12} \widehat{u}(t,\xi) \|_{L^2_\xi} \leq \epsilon.
\end{align*}
This leads to \eqref{eta2.2a}, and we complete the proof of Proposition \ref{S3prop1}.
\end{proof}

Let us present the proof of Lemma \ref{S3lem1}.

\begin{proof}[Proof of Lemma \ref{S3lem1}] We first observe that
\beq \label{S3eq6}
\begin{split}
\|t^{\f{\d-1}2}e^{t\D}u_0\|_{L^2(\R^+;\dot H^{\d+\f12})}^2=&\int_{\R^3}|\xi|^{1+2\d}\int_0^\infty t^{\d-1}e^{-2t|\xi|^2}\,dt
|\wh{u_0}(\xi)|^2\,d\xi\\
=&\int_{\R^3}\int_0^\infty \tau^{\d-1}e^{-2\tau}\,d\tau |\xi||\wh{u_0}(\xi)|^2\,d\xi=C\|u_0\|_{\dot H^{\f12}}^2.
\end{split}
\eeq
{This estimate will be used in the nonlinear computation below.}

By taking the $\dot H^{\d-\f12}$ inner product of \eqref{S2eq14} with $u=v$ and then multiplying the
resulting inequality by $t^{\d-1},$ we obtain that
\begin{align*}
\f12&\f{d}{dt}\bigl\|t^{\f{\d-1}2}\cB(u,u)(t)\bigr\|_{\dot H^{\d-\f12}}^2+\bigl\|t^{\f{\d-1}2}\cB(u,u)(t)\bigr\|_{\dot H^{\d+\f12}}^2
+(1-\d)\bigl\|t^{\f{\d}2-1}\cB(u,u)(t)\bigr\|_{\dot H^{\d-\f12}}^2\\
=& t^{\d-1}\bigl(Q(u,u), \cB(u,u)\bigr)_{\dot H^{\d-\f12}}\\
\leq &Ct^{\f{\d-1}2}\bigl\|(e^{t\D}u_0+\cB(u,u))\otimes (e^{t\D}u_0+\cB(u,u))\bigr\|_{\dot H^{\d-\f12}}\bigl\|t^{\f{\d-1}2}\cB(u,u)(t)\bigr\|_{\dot H^{\d+\f12}}\\
\leq &Ct^{{\d-1}}\bigl\|(e^{t\D}u_0+\cB(u,u))\otimes (e^{t\D}u_0+\cB(u,u))\bigr\|_{\dot H^{\d-\f12}}^2
+\f12\bigl\|t^{\f{\d-1}2}\cB(u,u)(t)\bigr\|_{\dot H^{\d+\f12}}^2,
\end{align*}
from which and the law of product in Sobolev spaces, we infer
\begin{align*}
\f{d}{dt}&\bigl\|t^{\f{\d-1}2}\cB(u,u)(t)\bigr\|_{\dot H^{\d-\f12}}^2+\bigl\|t^{\f{\d-1}2}\cB(u,u)(t)\bigr\|_{\dot H^{\d+\f12}}^2
+(1-\d)\bigl\|t^{\f{\d}2-1}\cB(u,u)(t)\bigr\|_{\dot H^{\d-\f12}}^2\\
\lesssim & \bigl\|t^{\f{\d-1}2}e^{t\D}u_0\|_{\dot H^{\d+\f12}}^2\bigl( \|e^{t\D}u_0\|_{\dot H^{\f12}}^2+\|\cB(u,u)\|_{\dot H^{\f12}}^2\bigr)
+\|\cB(u,u)\|_{\dot B^{\f32}}^2\bigl\|t^{\f{\d-1}2}\cB(u,u)(t)\bigr\|_{\dot H^{\d-\f12}}^2.
\end{align*}
By applying Gronwall's inequality and using \eqref{S3eq3}, \eqref{S3eq3a}, we find for $t\leq T(u_0),$
\begin{align*}
\bigl\|t^{\f{\d-1}2}\cB(u,u)(t)\bigr\|_{L^2_t(\dot H^{\d+\f12})}^2\leq
&C\bigl( \|e^{t\D}u_0\|_{L^\infty_t(\dot H^{\f12})}^2+\|\cB(u,u)\|_{L^\infty_t(\dot H^{\f12})}^2\bigr)\\
&\times \bigl\|t^{\f{\d-1}2}e^{t\D}u_0\|_{L^2_t(\dot H^{\d+\f12})}^2\exp\bigl(C\|\cB(u,u)\|_{L^2_t(\dot B^{\f32})}^2\bigr)\\
\leq & C(\|u_0\|_{\dot H^{\f12}}),
\end{align*}
which together with \eqref{S3eq6} implies that
\beq \label{S3eq7}
\|t^{\f{\d-1}2}u\|_{L^2_t(\dot H^{\d+\f12})}\leq  C(\|u_0\|_{\dot H^{\f12}}). \eeq
 We remark that estimate of this type was first proposed by Chemin and Planchon in \cite{CP12} for the classical 3-D Navier-Stokes
system. The avid reader may view it as a natural $L^2_t$ version of the classical Kato spaces (see \cite{kato}).

On the other hand, we get, by taking $\dot H^{\d+\f12}$ inner product of $(GNS)$ with $u$ and using
the law of product in Sobolev spaces, that
\begin{align*}
\f12\f{d}{dt}\|u(t)\|_{\dot H^{\d+\f12}}^2+\|\na u\|_{\dot H^{\d+\f12}}^2
=&\bigl(Q(u,u), u\bigr)_{\dot H^{\d+\f12}}\\
\lesssim &\|u\otimes u\|_{\dot H^{\d+\f12}}\|\na u\|_{\dot H^{\d+\f12}}\\
\lesssim& \|u\|_{L^\infty}\|u\|_{\dot H^{\d+\f12}}\|\na u\|_{\dot H^{\d+\f12}}.
\end{align*}
Multiplying $t^\d$ to the above inequality yields
\begin{align*}
\f{d}{dt}\|t^{\f\d2}u(t)\|_{\dot H^{\d+\f12}}^2+\|t^{\f\d2}\na u\|_{\dot H^{\d+\f12}}^2
\leq \d\|t^{\f{\d-1}2}u(t)\|_{\dot H^{\d+\f12}}^2+C\|u\|_{L^\infty}^2\|t^{\f\d2}u\|_{\dot H^{\d+\f12}}^2.
\end{align*}

Applying Gronwall's inequality gives rise to
\begin{align*}
\|t^{\f\d2}u\|_{L^\infty_t(\dot H^{\d+\f12})}^2+\|t^{\f\d2}\na u\|_{L^2_t(\dot H^{\d+\f12})}^2\leq C\|t^{\f{\d-1}2}u\|_{L^2_t(\dot H^{\d+\f12})}^2
\exp\bigl(C\|u\|_{L^2_t(L^\infty)}^2\bigr)\quad\mbox{for}\ t\leq T(u_0),
\end{align*}
which together with  \eqref{S3eq4} and \eqref{S3eq7} ensures \eqref{S3eq6a}. This completes the proof of Lemma \ref{S3lem1}.
\end{proof}

{By analogy of the corresponding norm of $\|\cdot\|_{X_T}$ defined in \eqref{S1eq3} for the subcritical case}, for sufficiently small positive
constant $\d,$
we define
\beq\label{S3eq9}
\|u \|_{Y_T}\eqdefa \bigl\| t^{\frac {\delta } 2} | \xi |^{\delta+\frac 12}  1_{|\xi| \ge   T^{-\frac 14} }
{e^{-\frac {\lambda^2 t}{4T}+\lambda \frac t {\sqrt T} |\xi|}   \widehat u (t,\xi)} 
\bigr\|_{L_T^{\infty}(L_{\xi}^2)}.
\eeq
{ Note here we introduce the special cut-off $1_{|\xi| \ge T^{-\frac 14}}$ to break the scaling.}
In what follows,
we shall focus on the estimating  of the
the $Y_T$-norm of $u$.

\begin{prop}\label{S3prop2}
{\sl Let $u_0 \in \dot H^{\frac 12}.$ There exists a sufficiently small positive constant $t_1$ so that
the system
$(GNS)$ with initial data $u_0$ has a unique solution $u$ on $[0,t_1]$ and
\beq\label{S3eq9b}
\| u\|_{Y_{T}} \leq C e^{10^{-4} \lambda^2(T)}
\|u_{\rm h}\|_{L_T^{\infty}(\dot H^{\frac 12}) }  \quad\mbox{for any}\ T\leq t_1,
\eeq where $\la(T)$ is defined by \eqref{S1eq9} and $u_{\rm h}$ is defined by \eqref{S3eq10} below.}
\end{prop}

\begin{proof} For simplicity, we just present the {\it a priori} estimates.
We take $M\eqdefa T^{-\frac 14}$ and split the solution $u$ of $(GNS)$ as
\beq\label{S3eq10}
\begin{split}
&u=u_{\rm l}+u_{\rm h} \with  \widehat{u_{\rm l}}
\eqdefa \underbrace{\widehat{u}\cdot 1_{|\xi| <M}} \andf \widehat{u_{\rm h}}\eqdefa
\underbrace{\widehat{u} \cdot 1_{|\xi|\ge M }},\\
&\wh{\cS(u)}(t,\xi)\eqdefa e^{-\frac {\lambda^2 t}{4T}+\lambda \frac t {\sqrt T} |\xi|}   \widehat{u} (t,\xi).
\end{split}\eeq

We first observe  from Proposition \ref{S3prop1} that for all $0<t\le T\leq 1$:
\beq \label{S3eq11}
\begin{split}
   t^{\frac {\delta} 2} e^{\frac {\lambda^2 t} {4T } }\| |D|^{\delta+\frac 12} \cS(u_{\rm l})(t) \|_{L^2}
 \lesssim & t^{\frac {\delta}2} \| |\xi|^{\delta+\frac 12} e^{\lambda \frac t {\sqrt T } |\xi|} 1_{|\xi| <
 T^{-\frac 14}}
 \widehat u (t, \xi) \|_{L^2_\xi}  \\
  \lesssim & t^{\frac {\delta}2}T^{-\f\d4} e^{\la T^{\f14}} \| |\xi|^{\frac 12}
 \widehat u (t, \xi) \|_{L^2_\xi}  \\
 \lesssim & T^{\frac {\delta}4} e^{\la T^{\f14}} C(\| u_0 \|_{\dot H^{\frac 12}}).
 \end{split}\eeq
 Note that due to our low frequency cut-off $1_{|\xi|\le T^{-\frac 14}},$ there is a saving of $T^{\frac {\delta} 4}$ in the
 {low frequency} estimate
 \eqref{S3eq11}.

In view of \eqref{S3eq9}, the  the $Y_T$-norm of $u$ is just  $\|t^{\frac {\delta} 2} |D|^{\delta+\frac 12} \cS(u_{\rm h}) \|_{L_T^{\infty}(L^2)}$.
To estimate $\| u \|_{Y_T}$, we first notice that
\beq\label{S3eq12}
\begin{split}
 & \bigl\|t^{\frac {\delta } 2} 1_{M \le |\xi| \le 3M }
  | \xi |^{\delta+\frac 12} e^{
 -\frac {\lambda^2 t} {4T} +\lambda \frac t {\sqrt T} |\xi| } \widehat u \bigr\|_{L^\infty_T(L^2_\xi)}  \\
 &\leq T^{\f\d2}M^\d\bigl\|e^{-t\left(\f\la{2\sqrt{T}}-|\xi|\right)^2}e^{9tM^2} |\xi|^{\f12} 1_{ |\xi| \geq M }  \widehat u \bigr\|_{L^\infty_T(L^2_\xi)}
 \\
 &\lesssim  T^{\frac {\delta} 4} \| u_{\rm h}\|_{L_t^{\infty}(H^{\frac12})}.
 \end{split}\eeq

To complete the estimate of $\|u\|_{Y_T}$, it remains for us to  estimate the term below: $$\bigl\| 1_{|\xi| \ge 3M }  t^{\frac {\delta } 2}
  | \xi |^{\delta+\frac 12} e^{
 -\frac {\lambda^2 t} {4T} +\lambda \frac t {\sqrt T} |\xi| } \widehat u \bigr\|_{L^\infty_T(L^2)}.$$
 For this, as in the proof of Theorem \ref{thm1}, we shall turn to nonlinear estimates. Again thanks to the frequency cut-off $1_{|\xi|\ge 3M}$, there will be no low-low interactions in the nonlinear estimate below.

Indeed similar to the proof of Proposition \ref{S2prop1}, for any $\eta_0>0,$ which will be chosen sufficiently small later on,
we split the integral $\int_0^t = \int_0^{\eta_0 t} + \int_{\eta_0 t}^t$
 and estimate
each pieces separately.\\

\noindent{\bf Step 1.} The estimate of the {short-time} piece $\int_0^{\eta_0 t}.$\\

Recall that $N_1=\la T^{-\f12},$ we write
\beq \label{S3eq13}
\begin{split}
 &\bigl\| t^{\frac {\delta} 2} |\xi|^{\frac 12+\d}  1_{|\xi| \ge 3M }
 \int_0^{\eta_0 t}  e^{-(t-s) |\xi|^2}
 e^{\lambda \frac t {\sqrt T } |\xi| -\frac {\lambda^2 t}{4 T} }
 \widehat{Q(u.u)}(s,\xi) \,ds \bigr\|_{L^\infty_T(L^2_\xi)} \\
 &\lesssim   \bigl\| t^{\frac {\delta} 2}
 1_{|\xi| < 2 N_1}
 \int_0^{\eta_0 t} |\xi|^{\frac 32+\delta}  e^{s |\xi|^2}
 1_{|\xi|\ge 3M } \widehat{u\otimes u}(s, \xi)\, ds \bigr\|_{L^\infty_T(L^2_\xi)}  \\
 & \qquad + \bigl\| t^{\frac {\delta} 2}   1_{|\xi| \ge 2 N_1 }
 \int_0^{\eta_0 t} |\xi|^{\frac 32+\delta}  e^{- \frac 1{10} t|\xi|^2}
 \widehat{ u\otimes u}(s,\xi) \,ds \|_{L^\infty_T(L^2_\xi)},
 \end{split}\eeq
 {where $\eta_0>0$ will be taken sufficiently small.}

We first deduce from Young's inequality and \eqref{S2eq5} that
\begin{align*}
 &  t^{\frac {\delta} 2}\int_{0}^{\eta_0t}   \bigl\|  1_{|\xi| < 2N_1}   |\xi|^{\frac 32 +\delta}e^{s|\xi|^2} 1_{|\xi|\ge 3M} \widehat{u\otimes u}(s, \xi) \bigr\|_{L_{\xi}^2} \,ds \\
 & \lesssim t^{\frac {\delta} 2} \int_{0}^{\eta_0 t} e^{4 s N_1^2} N_1^{\frac 32+\delta}\|1_{|\xi| \le 2N_1} \|_{L^6_\xi}\|1_{|\xi|\ge 3M} \widehat{u\otimes u}(s, \xi)\|_{L_{\xi}^3} \,ds\\
 & \lesssim t^{\frac {\delta} 2} \int_{0}^{\eta_0 t} e^{4 s N_1^2} N_1^{\frac 32+\delta}
  \cdot N_1^{\frac 12}  \bigl(\| u_{\rm l} \otimes u_{\rm h}(s)\|_{L^{\frac 3{2}}} +\|  u_{\rm h}\otimes u_{\rm h}(s) \|_{L^{\frac 3{2 }}}\bigr)\, ds  \\
& \lesssim (N_1^2 T)^{1+\frac{\delta} 2} e^{4\eta_0 T N_1^2}   \| u\|_{L^\infty_T(L^3)} \|u_{\rm h}\|_{L_T^{\infty}(L^3)} \\
& \lesssim  e^{5\eta_0 \lambda^2}\|u\|_{L^\infty_T(\dot H^{\frac 12})}
 \|u_{\rm h}\|_{L_T^{\infty}(\dot H^{\frac 12})},
   \end{align*}
   where in the last step, we {invoke the usual Sobolev embedding}: $\dot H^{\f12}(\R^3)\hookrightarrow L^3(\R^3).$

Along the same line, {we compute}
 \begin{align*}
 &t^{\frac {\delta} 2} \bigl\|
 \int_{0}^{\eta_0 t}  |\xi|^{\frac 32+\delta} 1_{|\xi| \ge  2N_1}e^{-   \frac {t}{10}  |\xi|^2}
 \widehat{ u\otimes u}(s,\xi)\, ds \bigr\|_{L^2_\xi}  \\
 &\lesssim t^{\frac {\delta} 2}
 \int_{0}^{\eta_0 t}  \bigl\||\xi|^{\frac 32+\delta} e^{-   \frac {t}{10}  |\xi|^2}\bigr\|_{L^6_\xi}
 \bigl\|1_{|\xi| \ge  2N_1}\widehat{ u\otimes u}(s,\xi) \|_{L^3_\xi} \, ds \\
&\lesssim   t^{-1}  \int_0^{\eta_0 t}
 \bigl (\| u_{\rm l}\otimes u_{\rm h}(s)\|_{L^{\frac 3{2}}} +\|  u_{\rm h}\otimes u_{\rm h}(s)\|_{L^{\frac 3{2 }}}\bigr) \,ds \\
 &\lesssim \|u\|_{L^\infty_T(\dot H^{\frac 12})}
 \|u_{\rm h}\|_{L_T^{\infty}(\dot H^{\frac 12})}.
  \end{align*}

By inserting the above estimates into \eqref{S3eq13}, we obtain
\beq \label{S3eq14}
\begin{split}
 \bigl\| t^{\frac {\delta} 2} |\xi|^{\frac 12+\d}  1_{|\xi| \ge 3M }
 \int_0^{\eta_0 t}  e^{-(t-s) |\xi|^2}
 e^{\lambda \frac t {\sqrt T } |\xi| -\frac {\lambda^2 t}{4 T} }
& \widehat{Q(u,u)}(s,\xi) \,ds \bigr\|_{L^\infty_T(L^2_\xi)} \\
& \lesssim  e^{5\eta_0 \lambda^2}\|u\|_{L^\infty_T(\dot H^{\frac 12})}
 \|u_{\rm h}\|_{L_T^{\infty}(\dot H^{\frac 12})}.
 \end{split}
 \eeq

 \noindent{\bf Step 2.} The estimate of the piece $\int_{\eta_0 t}^t.$\\

In view  of \eqref{S1eq1} and \eqref{S3eq10}, we have
\beq\label{S3eq15}
 \begin{split}
 & t^{\frac {\delta} 2} \bigl||\xi|^{\frac 12+\d}  1_{|\xi| \ge 3M }
 \int_{\eta_0 t}^{t}   e^{-(t-s) |\xi|^2}
 e^{\lambda \frac t {\sqrt T } |\xi| -\frac {\lambda^2 t}{4 T} }
 \widehat{Q(u,u)}(s,\xi) \,ds \bigr|\\
 &\lesssim
    t^{\frac {\delta} 2} |\xi|^{\f32+\d}
 \int_{\eta_0 t}^{t}   e^{- (t-s) |\xi|^2}
 e^{\lambda \frac {t-s} {\sqrt T} |\xi| +\frac {\lambda^2 s} {2 T}-\frac {\lambda^2 t}{4T} }
 1_{|\xi| \ge 3M }\bigl|\widehat{\cS(u)\otimes\cS(u)}(s,\xi)\bigr|\,ds.
 \end{split}\eeq
 By frequency localization \eqref{S2eq5}, we find
 \begin{align*}
 1_{|\xi| \ge 3M }\widehat{\cS(u)\otimes\cS(u)} = 1_{|\xi|\ge 3M }
 \widehat{\cS(u_{\rm l})\otimes \cS(u_{\rm h})}  +  1_{|\xi|\ge 3M }
 \widehat{\cS(u_{\rm h})\otimes\cS(u_{\rm l})}
 +
 1_{|\xi|\ge 3M } \widehat{\cS(u_{\rm h})\otimes\cS(u_{\rm h})}.
 \end{align*}
 We first handle the contribution due to $\cS(u_{\rm h})\otimes\cS(u_{\rm h})$. Indeed, it is easy to observe that
\begin{align*}
  &\bigl\| t^{\frac {\delta} 2} |\xi|^{\f32+\d}
 \int_{\eta_0 t}^{t}    e^{- (t-s) |\xi|^2}
 e^{\lambda \frac {t-s} {\sqrt T} |\xi| +\frac {\lambda^2 s} {2 T}-\frac {\lambda^2 t}{4T} }
 \widehat{\cS(u_{\rm h})\otimes\cS(u_{\rm h})}(s,\xi)\, ds \|_{L^\infty_T(L^2_\xi)} \\
 & \lesssim \bigl\| t^{\frac {\delta} 2}
 \int_{\eta_0 t}^{t} |\xi|^{\frac 32+\delta} e^{-   (t-s) \left(|\xi| - \frac {\lambda } {2\sqrt T } \right)^2  +\frac {\lambda^2 s}{4T}}
 \widehat{\cS(u_{\rm h})\otimes\cS(u_{\rm h})}(s,\xi)\,ds \|_{L^\infty_T(L^2_\xi)} \\
  &\lesssim  \bigl\| t^{\frac {\delta} 2}
 \int_{\eta_0 t}^{t} |\xi|^{\frac 32+\delta} 1_{|\xi| \le N_1} e^{\frac {\lambda^2 s}{4T} }
  \widehat{\cS(u_{\rm h})\otimes\cS(u_{\rm h})}(s,\xi)\, ds \|_{L^\infty_T(L^2_\xi)}  \\
  &\quad + \bigl\| t^{\frac {\delta} 2}
 \int_{\eta_0 t}^{t} |\xi|^{\frac 32+\delta} 1_{|\xi| \ge  N_1 }e^{-   \frac 1{10} (t-s) |\xi|^2
 +\frac {\lambda^2 s }{4T} }
 \widehat{\cS(u_{\rm h})\otimes\cS(u_{\rm h})}(s,\xi)\,ds \|_{L^\infty_T(L^2_\xi)},
 \end{align*}
 from which, $N_1\geq 1$ and Lemma \ref{S2lem1}, we infer
 \begin{align*}
  &\bigl\| t^{\frac {\delta} 2} |\xi|^{\f32+\d}
 \int_{\eta_0 t}^{t}    e^{- (t-s) |\xi|^2}
 e^{\lambda \frac {t-s} {\sqrt T} |\xi| +\frac {\lambda^2 s} {2 T}-\frac {\lambda^2 t}{4T} }
 \widehat{\cS(u_{\rm h})\otimes\cS(u_{\rm h})}(s,\xi)\, ds \|_{L^\infty_T(L^2_\xi)} \\
 &\lesssim  \lambda^{-\delta} e^{\frac {\lambda^2 }4} \cdot \sup_{0<s \le T}
 \bigl(s^{\delta}\| \cS(u_{\rm h}) (s)  \cS(u_{\rm h})(s) \|_{L^\infty_T(L^{\frac 3 {2(1-\delta)} })}  \bigr)  \\
& \lesssim  \lambda^{-\delta} e^{\frac {\lambda^2 }4} \cdot  \sup_{0<s \le T}
 \Bigl(s^{\delta} \| |D|^{\frac 12 +\delta} \cS(u_{\rm h})(s)\|_{L^\infty_T(L^2)}^2\Bigr) \\
& \lesssim   \lambda^{-\delta} e^{\frac {\lambda^2 } 4} \|u\|_{Y_T}^2,
 \end{align*}
 where in the last step, we {applied the usual Sobolev embedding } $\dot H^{\f12+\d}(\R^3)\hookrightarrow L^{\frac 3 {(1-\delta)} }(\R^3).$

Along the same line, we deduce that
  \begin{align*}
 & \bigl\| t^{\frac {\delta} 2} |\xi|^{\f32+\d}
 \int_{\eta_0 t}^{t}    e^{- (t-s) |\xi|^2}
 e^{\lambda \frac {t-s} {\sqrt T} |\xi| +\frac {\lambda^2 s} {2 T}-\frac {\lambda^2 t}{4T} }
 \widehat{\cS(u_{\rm l})\otimes\cS(u_{\rm h})}(s,\xi)\,ds \|_{L^\infty_T(L^2_\xi)} \\
 &\lesssim  \bigl\| t^{\frac {\delta} 2}
 \int_{\eta_0 t}^{t} |\xi|^{\frac 32+\delta} e^{-   (t-s) \left(|\xi| - \frac {\lambda } {2\sqrt T } \right)^2  +\frac {\lambda^2 s}{4T}}
 \widehat{\cS(u_{\rm l})\otimes\cS(u_{\rm h})}(s,\xi)\,ds \bigr\|_{L^\infty_T(L^2_\xi)} \\
  &\lesssim \bigl\| t^{\frac {\delta} 2}
 \int_{\eta_0 t}^{t} |\xi|^{\frac 32+\delta} 1_{|\xi| \le N_1} e^{\frac {\lambda^2 s}{4T} }
  \widehat{\cS(u_{\rm l})\otimes\cS(u_{\rm h})}(s,\xi)\,ds \|_{L^\infty_T(L^2_\xi)} \\
  &\quad + \bigl\| t^{\frac {\delta} 2}
 \int_{\eta_0 t}^{t} |\xi|^{\frac 32+\delta} 1_{|\xi| \ge  N_1 }e^{-   \frac 1{10} (t-s) |\xi|^2
 +\frac {\lambda^2 s }{4T} }
 |\widehat{\cS(u_{\rm l})\otimes\cS(u_{\rm h})}(s,\xi)\, ds \|_{L^\infty_T(L^2_\xi)},
 \end{align*}
 from which and Lemma \ref{S2lem1}, we get, by a similar derivation of \eqref{S2eq11a}, that
 \begin{align*}
 &\bigl\| t^{\frac {\delta} 2} |\xi|^{\f32+\d}
 \int_{\eta_0 t}^{t}    e^{- (t-s) |\xi|^2}
 e^{\lambda \frac {t-s} {\sqrt T} |\xi| +\frac {\lambda^2 s} {2 T}-\frac {\lambda^2 t}{4T} }
 \widehat{\cS(u_{\rm l})\otimes\cS(u_{\rm h})}(s,\xi)\,ds \|_{L^\infty_T(L^2_\xi)} \\
 &\lesssim  \bigl(\lambda^{2-\delta} +1\bigr) \cdot \sup_{0<s \le T} \bigl(s^{\delta}
 e^{\frac {\lambda^2 s}{4T} }\| \cS(u_{\rm l})(s)  \cS(u_{\rm h})(s) \|_{L^\infty_T(L^{\frac 3 {2(1-\delta)} })}\bigr) \\
 &\lesssim  \lambda^{2-\delta}  \cdot  \sup_{0<s \le T}
 \bigl(s^{\frac {\delta}2}  e^{\frac {\lambda^2 s}{4T} } \| |D|^{\frac 12+\delta} \cS(u_{\rm l})(s) \|_{L^2}
 \cdot s^{\frac {\delta}2} \| |D|^{\frac 12 +\delta} \cS(u_{\rm h})(s) \|_{L^2}\bigr)  \\
 &\lesssim   \lambda^{2-\delta} T^{\frac {\delta}4}e^{\la T^{\f14}}\|u_0\|_{\dot H^{\frac 12}} \|u\|_{Y_T},
 \qquad (\text{by \eqref{S3eq11}}).
 \end{align*}

{Plugging the above estimates} into \eqref{S3eq15}, we achieve
\beq\label{S3eq16}
 \begin{split}
 &\bigl\| t^{\frac {\delta} 2} |\xi|^{\frac 12+\d}  1_{|\xi| \ge 3M }
 \int_{\eta_0 t}^{t}   e^{-(t-s) |\xi|^2}
 e^{\lambda \frac t {\sqrt T } |\xi| -\frac {\lambda^2 t}{4 T} }
 \widehat{Q(u,u)}(s,\xi) \,ds \bigr\|_{L^\infty_T(L^2_\xi)}\\
 &\lesssim   \lambda^{-\delta} e^{\frac {\lambda^2 } 4} \|u\|_{Y_T}^2+\lambda^{2-\delta} T^{\frac {\delta}4}e^{\la T^{\f14}} \|u_0\|_{\dot H^{\frac 12}}\|u\|_{Y_T}.
 \end{split}
 \eeq

 \noindent{\bf Step 3.} The estimate of $\|u\|_{Y_T}.$\\

Now let us return to the estimate of \eqref{S3eq9b}.
Indeed by virtue of \eqref{S2eq20} and \eqref{S3eq9}, we get, by summing up the estimates
\eqref{S3eq12}, \eqref{S3eq14} and \eqref{S3eq16}, that
\beq\label{S3eq17}
\begin{split}
\|u\|_{Y_T} \le \| e^{t\Delta} u_0 \|_{Y_T}  +  C_1\bigl(&
T^{\frac {\delta} 4} \| u_{\rm h}\|_{L_t^{\infty}(\dot H^{\frac12})} +
e^{5\eta_0 \lambda^2}  \| u_0\|_{\dot H^{\frac 12}} \|u_{\rm h}\|_{L_t^{\infty}(\dot H^{\frac 12})}
 \\
&  + \lambda^{-\delta}e^{\frac {\lambda^2}{4}}
 \|u\|_{Y_T}^2 + T^{\frac {\delta} 4}  \lambda^{2-\delta} e^{\la T^{\f14}} \|u_0\|_{\dot H^{\frac 12}}\|u\|_{Y_T}\bigr).
\end{split}\eeq

To see the smallness of the linear term $\| e^{t\Delta} u_0 \|_{Y_T}$, we observe that for all $0<t\le T$:
\begin{align*}
\| e^{t\Delta} u_0 \|_{Y_T}= &\| t^{\frac {\delta} 2} |\xi|^{\frac 12+\delta} 1_{|\xi| \ge    T^{-\frac 14} }
e^{-\frac {\lambda^2 t}{4T} +\lambda \frac t {\sqrt T } |\xi| } e^{-t|\xi|^2} \widehat {u_0} (\xi) \|_{L^\infty_T(L^2_\xi)}\\
\lesssim &
\bigl\| 1_{|\xi |\ge 2N_1 }|\xi|^{\frac 12}
t^{\frac {\delta} 2} |\xi|^{\delta} e^{-\f1{10} t |\xi|^2}
 \widehat {u_0} (\xi) \bigr\|_{L^\infty_T(L^2_\xi)} \\
 & +
 \bigl\| t^{\frac {\delta} 2} |\xi|^{\frac 12+\delta}  1_{|\xi| \le 2 N_1}
  1_{|\xi| \ge T^{-\frac 14}} \widehat {u_0} (\xi) \bigr\|_{L^\infty_T(L^2_\xi)}
  \\
\lesssim &
\| u_{\rm h} \|_{L_t^{\infty}(\dot H^{\frac 12})} + \lambda^{\delta} \| u_{\rm h} \|_{L_t^{\infty}(\dot H^{\frac 12}) }
\le C_1  \lambda^{\delta} \| u_{\rm h} \|_{L_T^{\infty}(\dot H^{\frac 12})}.
\end{align*}

Let us now take $\la=\la(T),$  which is given by \eqref{S1eq9}.
Then by shrinking $\mathcal{T}>0$ to be so small  that
\beq \label{S3eq18}
\mathcal{T}\eqdefa \bigl\{ T:\ C_1 \|u_0 \|_{H^{\frac 12}} T^{\frac {\delta} 4} \cdot \lambda^{2-\delta}(T)e^{\la(T) T^{\f14}} \leq \frac 12 \ \bigr\},
\eeq
we deduce from \eqref{S3eq17} that for any $T\leq \mathcal{T},$
\beq\label{S3eq19}
\|u\|_{Y_T} \le  C_2\bigl( e^{5\eta_0 \lambda^2(T)}
\|u_{\rm h}\|_{L_T^{\infty}(\dot H^{\frac 12}) }
+\lambda^{-\delta}(T)e^{\frac {\lambda^2(T)}{4}}
 \|u\|_{Y_T}^2\bigr).
\eeq
We further shrink $t_1>0$ to be so small that
\beq\label{S3eq20}
t_1\eqdefa \bigl\{T\leq \mathcal{T}: 4C_2^2
\|u_{\rm h}\|_{L_T^{\infty}(\dot H^{\frac 12}) }  e^{5\eta_0 \lambda^2(T)}
\cdot \lambda^{-\delta}(T) e^{\frac {\lambda^2(T)}{4}}
  \leq \frac 1{2}\ \bigr\}.
\eeq
We remark that thanks to the definition of $\la(T)$ given by \eqref{S1eq9} and \eqref{eta2.2a}, $t_1$ defined by \eqref{S3eq20} can be reached.

Then we deduce from \eqref{S3eq19} that for any $T\leq t_1,$
\beq \label{S3eq21}
\|u\|_{Y_T} \leq 2 C_2 e^{5\eta_0 \lambda^2(T)}
\|u_{\rm h}\|_{L_T^{\infty}(\dot H^{\frac 12}) }.
\eeq
In particular, by taking $\eta_0=10^{-5},$ we conclude the proof of  \eqref{S3eq9b}.
We thus complete the proof of Proposition \ref{S3prop2}.
\end{proof}

Now we are in a position to complete the proof of Theorem \ref{thm2}.

\begin{proof}[Proof of Theorem \ref{thm2}] Again for simplicity, we just present the {\it a priori} estimates.
In view of \eqref{S3eq10}, we get, by summing up \eqref{S3eq9b} and \eqref{S3eq11},
that for any $T\leq t_1,$
\begin{align*}
\|t^{\f{\d}2}\cS(u)\|_{L^\infty_T(\dot H^{\f12+\d})}\leq &\|t^{\f{\d}2}\cS(u_{\rm l})\|_{L^\infty_T(\dot H^{\f12+\d})}
+\|u\|_{Y_T}\\
\leq &C(\|u_0\|_{\dot H^{\f12}})T^{\f{\d}4}e^{-\f{\la^2(T)}{4}}+Ce^{10^{-4}\la^2(T)}\|u_{\rm h}\|_{L^\infty_T(\dot H^{\f12})}.
\end{align*}
In particular, by taking $t=T$ in the above inequality, we achieve for $T\leq t_1,$ which is determined by Proposition \ref{S3prop2},
\beno
T^{\f\d2}e^{-\f{\la^2(T)}4}\|e^{\la(T)\sqrt{T}|D|}u(T)\|_{\dot H^{\f12+\d}}\leq C(\|u_0\|_{\dot H^{\f12}})T^{\f{\d}4}e^{-\f{\la^2(T)}{4}}+Ce^{10^{-4}\la^2(T)}\|u_{\rm h}\|_{L^\infty_T(\dot H^{\f12})},
\eeno
which leads to \eqref{S1eq10q}. We thus complete the proof of Theorem \ref{thm2}.
\end{proof}

\medskip

\section*{Acknowledgments}
This work was initiated during P. Zhang's visit to the University of Hong Kong in the spring of 2024.
He would to appreciate the hospitality of the colleagues in the department of mathematics.
D. Li is supported  in part by NSFC 12271236.
P. Zhang is supported by National Key R$\&$D Program of China under grant
  2021YFA1000800 and K. C. Wong Education Foundation.
 He is also partially supported by NSFC  12288201 and 12031006.


\medskip


\begin{thebibliography}{50}


\bibitem{BBT} H. Bae, A. Biswas and E. Tadmor,  Analyticity and decay estimates of the Navier-Stokes
 equations in critical Besov spaces, {\it Arch. Ration. Mech. Anal.}, {\bf 205} (2012),  963-991.

\bibitem{bcd}
 H. Bahouri, J.-Y. Chemin and R. Danchin,
{\it Fourier Analysis and Nonlinear Partial Differential Equations}, Grundlehren der mathematischen Wissenschaften, {\bf 343}, Springer-Verlag Berlin Heidelberg, 2011.

\bibitem{BJMT} A. Biswas, M. S. Jolly, V. R. Martinez, E. S. Titi,  Dissipation length scale estimates for turbulent flows: a Wiener algebra approach, {\it J. Nonlinear Sci.}, {\bf 24} (2014),  441-471.

\bibitem{BS07} A.    Biswas and D. Swanson,  Gevrey regularity of solutions to the 3-D Navier-Stokes equations with weighted $\ell_p$ initial data, {\it Indiana Univ. Math. J.}, {\bf 56} (2007), 1157-1188.


\bibitem{Bo} J.~M. Bony, Calcul symbolique et propagation des singularit\'es pour
les \'equations aux d\'eriv\'ees partielles non lin\'eaires, {\it
Ann. Sci. \'Ecole Norm. Sup.}, {\bf 14} (1981), 209--246.

\bibitem{BGK} Z. Bradshaw,  Z. Gruji\'c and I. Kukavica,  Local analyticity radii of solutions to the 3D Navier-Stokes equations with locally analytic forcing, {\it  J. Differential Equations}, {\bf  259} (2015), 3955-3975.

\bibitem{cheminleray} J.-Y. Chemin, Le syst\`eme de Navier-Stokes incompressible soixante dix ans apr\`es Jean Leray, {\it  Actes des Journ\'ees Math\'ematiques \`a la M\'emoire de Jean Leray,}  99-123,  S\'eminaire et Congr\`es,
{\bf 9}, Soci\'et\'e Math\'ematique de  France, Paris, 2004.


\bibitem{CGZ3}  J.-Y. Chemin, I. Gallagher and P. Zhang, On the radius of analyticity
 of solutions to semi-linear  parabolic system, {\it Math. Res. Lett.}, {\bf 27} (2020),  1631-1643.

 \bibitem{CP12}  J.-Y. Chemin and  F. Planchon,
Self-improving bounds for the Navier-Stokes equations,
{\it Bull. Soc. Math. France}, {\bf 140} (2012),  583-597.

\bibitem{DT95} C.~R. Doering and E.~S. Titi, Exponential decay rate of the power spectrum for solutions of the Navier-Stokes equations,
{\it Phys. Fluids}, {\bf 7} (1995),  1384-1390.

\bibitem{F97} C. Foias,  What do the Navier-Stokes equations tell us about turbulence? Harmonic analysis and nonlinear differential equations (Riverside, CA, 1995), 151-180, Contemp. Math., {\bf 208}, Amer. Math. Soc., Providence, RI, 1997.


\bibitem{FMRT}  C. Foias, O. Manley, R. Rosa and R. Temam, {\it  Navier-Stokes equations and turbulence.} Encyclopedia of Mathematics and its Applications, {\bf 83}. Cambridge University Press, Cambridge, 2001.

\bibitem{foiastemam} C. Foias and R. Temam, Gevrey class regularity for the solutions of the Navier-Stokes equations,
 {\it J. Funct. Anal.}, {\bf 87} (1989), 359-369.


 \bibitem{fujitakato}
H. Fujita and T. Kato, On the Navier-Stokes initial value problem I,
{\it  Arch. Ration. Mech. Anal.}, {\bf 16} (1964),
269--315.


\bibitem{Giga83}
 Y. Giga, Time and spatial analyticity of solutions of the Navier–Stokes equations, Comm. Partial Differential Equa-
tions 8 (8) (1983) 929--948.



\bibitem{GK98} Z. Gruji$\acute{c}$ and I. Kukavica,  Space analyticity for the Navier-Stokes and related equations with initial data in
$L^p,$ {\it  J. Funct. Anal.}, {\bf 152} (1998),  447-466.

\bibitem{HKR90} W.~D. Henshaw, H.~O. Kreiss and L. G. Reyna,  Smallest scale estimates for the Navier-Stokes equations for incompressible fluids,
{\it  Arch. Rational Mech. Anal.}, {\bf 112} (1990),  21-44.

\bibitem{HS} I. Herbst and  E. Skibsted, Analyticity estimates for the Navier-Stokes equations, {\it Adv.  Math.}, {\bf  228} (2011),  1990-2033.


\bibitem{H97}
L. H\"ormander, {\it Lectures on nonlinear hyperbolic differential equations}. Math\'ematiques $\&$ Applications (Berlin), {\bf  26}. Springer-Verlag, Berlin, 1997.





\bibitem{HZ1}
R. Hu and P. Zhang, On the radius of analyticity of solutions to 3D Navier-Stokes  system with initial data in $L^p,$ {\it  Chin. Ann. Math. Ser. B},
{\bf 43} (2022),  749-772.

\bibitem{kato} T. Kato, Strong $L^p$-solutions of the Navier-Stokes equation in $\R^m$ with applications to weak solutions, {\it Math. Z.}, {\bf  187} (1984), 471-480.

\bibitem{katomasuda} T. Kato and K. Masuda,
Nonlinear evolution equations and analyticity I,  {\it
Annales de l'IHP  section C}, {\bf 3} (1986),  455-467.

\bibitem{Ku99} I. Kukavica,  On the dissipative scale for the Navier-Stokes equation, {\it Indiana Univ. Math. J.}, {\bf 48} (1999), 1057-1081.

 \bibitem{lemarie2} P.-G. Lemari\'e-Rieusset,  Une remarque sur l'analyticit\'e des solutions milds des \'equations de Navier-Stokes dans $\R^3$, {\it C. R. Acad. Sci. Paris{,} } {\bf 330} (2000),  183-186.


 \bibitem{lemarie1} P.-G. Lemari\'e-Rieusset,  Nouvelles remarques sur l'analyticit\'e des solutions milds des \'equations de Navier-Stokes dans $\R^3$, {\it C. R. Math. Acad. Sci. Paris,} {\bf 338} (2004),  443-446.


 \bibitem{lemarie3} P.-G. Lemari\'e-Rieusset, {\it The Navier-Stokes problem in the 21st century}. CRC Press, Boca Raton, FL, 2016.


\bibitem{lerayns}
J. Leray, Essai sur le mouvement d'un liquide visqueux emplissant
l'espace, {\em Acta Math.}, {\bf 63} (1933),  193--248.


\bibitem{LiRMI19} D. Li, On Kato-Ponce and fractional Leibniz, {\it Rev. Mat. Iberoam.}, {\bf 35} (2019), 23--100.


 \bibitem{Ma67} K.  Masuda,  On the analyticity and the unique continuation theorem for solutions of the Navier-Stokes equation,
  {\it Proc. Japan Acad.}, {\bf 43} (1967), 827-832.







\bibitem{Zhang22} P. Zhang, On the instantaneous  radius of analyticity of $L^p$ solutions to 3D Navier-Stokes  system,
{\it Math. Z.}, {\bf 304} (2023), no. 3,  Paper No. 38, 32 pp.




\end{thebibliography}
\end{document}